\documentclass{amsart}
\usepackage{amsmath,amssymb,epsfig,amscd,amsthm}

\usepackage{verbatim}

\numberwithin{equation}{section}

\newtheorem{lemma}[equation]{Lemma}
\newtheorem{thm}[equation]{Theorem}

\newtheorem{cor}[equation]{Corollary}
\newtheorem{prop}[equation]{Proposition}

\newtheorem{defi}[equation]{Definition}

\theoremstyle{remark}

\newtheorem{remark}[equation]{Remark}

\renewcommand{\bar}[1]{#1\llap{$\overline{\phantom{\rm#1}}$}}

\newcommand{\lra}{\longrightarrow}

\DeclareMathOperator{\Prep}{{Prep}}
\newcommand{\hhat}{\widehat{h}}

\newcommand{\N}{{\mathbb N}}
\newcommand{\Z}{{\mathbb Z}}
\newcommand{\Q}{{\mathbb Q}}

\newcommand{\bK}{\overline{K}}

\newcommand{\Kbar}{{\bar{K}}}

\newcommand{\bC}{{\mathbb C}}

\newcommand{\cU}{\mathcal{U}}

\newcommand{\cZ}{\mathcal{Z}}

\newcommand{\cE}{\mathcal{E}}

\newcommand{\fp}{\mathfrak p}
\newcommand{\fq}{\mathfrak q}

\newcommand{\fo}{\mathfrak o}

\newcommand{\bP}{{\mathbb P}}

\begin{document}



\title{Portraits of preperiodic points for rational maps}

\author{D.~Ghioca}
\address{
Dragos Ghioca\\
Department of Mathematics\\
University of British Columbia\\
Vancouver, BC V6T 1Z2\\
Canada
}
\email{dghioca@math.ubc.ca}

\author{K.~Nguyen}
\address{
Khoa Nguyen
Department of Mathematics\\
University of British Columbia\\
Vancouver, BC V6T 1Z2\\
Canada
}
\email{dknguyen@math.ubc.ca}

\author{T.~J.~Tucker}
\address{
Thomas Tucker\\
Department of Mathematics\\ 
University of Rochester\\
Rochester, NY 14627\\
USA
}
\email{thomas.tucker@rochester.edu}

\thanks{The first author was partially supported by an NSERC Discovery Grant. The second author thanks the Pacific Institute of Mathematical Sciences for its generous support. The third
  author was partially supported by  an NSF grant.}


 \begin{abstract}
Let $K$ be a function field over an algebraically closed field $k$ of characteristic $0$, let $\varphi\in K(z)$ be a rational function of degree at least equal to $2$ for which there is no point at which $\varphi$ is totally ramified, and let $\alpha\in K$. We show that for all but finitely many pairs $(m,n)\in \Z_{\ge 0}\times \N$ there exists a place $\fp$ of $K$ such that the point $\alpha$ has preperiod $m$ and minimum period $n$ under the action of $\varphi$.  This answers a conjecture made by Ingram-Silverman \cite{IS} and Faber-Granville \cite{FG}. We prove a similar result, under suitable modification, also when $\varphi$ has points where it is  totally ramified. We give several applications of our result, such as  showing that for any tuple $(c_1,\dots , c_{d-1})\in k^{n-1}$ and for almost all pairs $(m_i,n_i)\in \Z_{\ge 0}\times \N$ for $i=1,\dots, d-1$, there exists a polynomial $f\in k[z]$ of degree $d$ in normal form such that for each $i=1,\dots, d-1$, the point $c_i$ has preperiod $m_i$ and minimum period $n_i$ under the action of $f$. 
\end{abstract}


\maketitle

\section{Introduction}
\label{intro}

Throughout this paper, let $K$ be a finitely generated field of transcendence
degree $1$ over an algebraically closed field $k$ of characteristic $0$. By a place
$\fp$
of $K$, we mean an equivalence class of valuations on $K$ that are trivial on $k$; the set of all such places is denoted by $\Omega_K$.  
Each such place $\fp$ gives rise to a valuation ring $\fo_\fp$
and a maximal ideal denoted (by an abuse of notation) by $\fp$. The residue field
is (canonically isomorphic to) $k$. We let
$v_\fp$ and $|\cdot|_\fp$ respectively denote the corresponding additive and multiplicative
valuations normalized by $v_\fp(K)=\Z$ and $|\alpha|_\fp=e^{-v_\fp(\alpha)}$
for every $\alpha\in K$. For a rational map $\varphi\in K(x)$, 
 for all but finitely many places $\fp$ of $K$ we can define
the reduction of $\varphi$ modulo $\fp$, see Section \ref{good reduction for rational maps} or \cite[Chapter 2]{Silverman-book}. We start with a definition which is central for our paper:

\begin{defi}
\label{portrait definition}
Let $F$ be any field. For a rational map $\varphi\in F(z)$ and for $\alpha\in F$, we say that $(m,n)\in \Z_{\ge 0}\times \N$ is the preperiodicity portrait, or simply portrait, of $\alpha$ (with respect to $\varphi$) if  $\varphi^{m}(\alpha)$ is periodic of minimum period $n$ for $\varphi$, while $m$ is the smallest nonnegative integer such that $\varphi^{m}(\alpha)$ is periodic (as always in dynamics, we denote by $\varphi^k=\varphi\circ \cdots \circ \varphi$ composed with itself $k$ times). We call $m$ the {\it preperiod} of
$\alpha$ and call $n$ the {\it minimum period} of $\alpha$. 
\end{defi}

Let $\varphi\in K(z)$, $\alpha\in \bP^1(K)$, and $\fp$ a place of $K$
such that the reduction of $\varphi$ modulo $\fp$ is
well-defined. Assume that $\alpha$ has portrait $(m,n)$ under the
induced reduction self-map on $\bP^1(k)$. We call $(m,n)$ the
preperiodicity portrait of $\alpha$ (under the action of $\varphi$)
\textit{modulo $\fp$}. The existence of places $\fp$ of $K$ for which
$\alpha$ has portrait $(m,n)$ under the action of $\varphi$ modulo
$\fp$ has been studied for more than 100 years (see the papers of Bang
\cite{Bang} and Zsigmondy \cite{Zsigmondy}) and also more recently
(see \cite{Schinzel, IS, FG, GNT, Krieger}). Some of the results are
easier to obtain when the dynamical system induces a group action (see
\cite{Bang, Zsigmondy}, and also more recent results for Drinfeld
modules \cite{equi, L-C}). However, when there is no group action, the
problem is much harder. Certain conjectures on this problem have been
made by Ingram-Silverman \cite[pp.~300--301]{IS} and modified by
Faber-Granville \cite[pp.~190]{FG}, yet very few general results are
known.  Faber and Granville
\cite{FG} have proved that if $\varphi\in\Q(z)$ and $\alpha\in\Q$, then for
all but finitely many $m\in\Z_{\ge 0}$ there exists a co-finite set
$S_m\subseteq \N$ (i.e., $\N\setminus S_m$ is finite) such that for
each $n\in S_m$, there exists a prime $p$ such that the preperiod of
$\alpha$ modulo $p$ is $m$, while its minimum period \emph{divides}
$n$. On the other hand, Ingram-Silverman \cite{IS} and Faber-Granville
\cite{FG} conjecture that one can obtain a similar conclusion this
time for minimum period \emph{equal} to $n$. In this paper, we are
able to resolve their conjecture over function fields. First we define
the set of exceptions from the conjecture of Ingram-Silverman
\cite{IS} and Faber-Granville \cite{FG}.

\begin{defi}
For any rational function $\varphi\in K(z)$, let $X(\varphi)$ be the set of $n$
such that $\varphi$ is totally ramified at every point of minimum period $n$. 
 For a rational function
$\varphi$ and a point $\alpha$, we let $Y(\varphi, \alpha)$ be the
set of positive integers $m$ such that $\varphi$ is
totally ramified over
$\varphi^m(\alpha)$ (i.e. $\varphi$ is totally ramified at 
$\varphi^{m-1}(\alpha)$). 
\end{defi}

Now we can state our main result; for more details on good reduction of a rational map $\varphi$ and for the canonical height $\hhat_\varphi$ associated to $\varphi$, see Section~\ref{subsect:canonical height}.
\begin{thm}\label{MN}
Let $K$ be a function field.   Let $\varphi \in K(z)$
have degree $d > 1$.  Let $\tau,s > 0$ and let $S$ be a set of
places of $K$ containing all the
places of bad reduction for $\varphi$ such that $\#S\leq s$.  Then there is a finite set
$\cZ(\tau, s)\subset \Z_{\ge 0}\times \N$ depending only on $\varphi$, $K$, $\tau$, and $s$ with the following property:  for any $\alpha \in \bP^1(K)$ with
$\hhat_\varphi(\alpha) > \tau$ and any $(m,n) \in \Z_{\ge 0}\times \N$ such that $(m,n) \not
\in \cZ(\tau,s)$, $m \notin Y(\varphi, \alpha)$ and $n \notin
X(\varphi)$, there is a nonarchimedean prime $\fp$ of good reduction
for $\varphi$ such that $\alpha$ has portrait $(m,n)$ under the action of $\varphi$ modulo $\fp$.
Moreover, the set $\cZ(\tau, s)$ is
effectively computable.  
\end{thm}

The next two remarks explain why the conditions in Theorem 
\ref{MN} are necessary.

\begin{remark}\label{rem:FG vs IS}
Theorem \ref{MN} \textit{without} the conditions that 
$m\notin Y(\varphi,\alpha)$ and $n\notin X(\varphi)$
was essentially conjectured by Ingram-Silverman \cite{IS}. It was Faber and Granville 
\cite{FG}
who pointed out that the condition $n\notin X(\varphi)$
would be necessary. By Lemma~\ref{inv},  $n\in X(\varphi)$ if and only if either $\varphi$ has no point of minimum period $n$ (see Kisaka's classification
\cite{Kisaka} or \cite[Appendix B]{FG} for a complete list of all rational maps which have no point of minimum period $n$), or $n=2$ and $\varphi(z)$ is
linearly conjugate to $z^{-2}$. Note that Ingram-Silverman made
their conjectures over number fields while Faber-Granville even restricted
further to the field of rational numbers. Theorem \ref{MN} answers completely
the same question for function fields.
\end{remark}

\begin{remark}\label{rem:we vs FG}
Theorem \ref{MN} \textit{without} the condition that $m\notin Y(\varphi,\alpha)$
was essentially conjectured by Faber-Granville \cite[pp.~190]{FG}. We
briefly explain why this condition is also necessary. Suppose $M> 0$
such that $\varphi$ is totally ramified over $\varphi^{M}(\alpha)$. Write
$\beta=\varphi^M(\alpha)$. We may assume $\varphi^{M-1}(\alpha)=0$
and $\beta\neq \infty$ 
by making a linear change of variables. Therefore $\varphi$ has the form:
$$\varphi(z)=\frac{z^d}{\psi(z)}+\beta$$
where $\psi$ is a polynomial of degree at most $d$ satisfying $\psi(0)\neq 0$.
For almost all primes $\fp$ of good reduction and for every $a\in \bP^1(K)$,
we have that $\varphi(a)\equiv \beta$ mod $\fp$ if and only if 
$a\equiv 0$ mod $\fp$. Hence if $\varphi^{M+n}(\alpha)\equiv \beta$ mod $\fp$
then $\varphi^{M+n-1}(\alpha)\equiv 0$ mod $\fp$. So, for almost
all $n$, there does not exist $\fp$ such that 
$\alpha$ has portrait $(M,n)$ under the action of $\varphi$ mod $\fp$.
\end{remark}

We also note that the hypothesis that the function field has
characteristic 0 is used crucially in our proof because we employ in
our arguments Mason's \cite{Mason} and Stothers' \cite{Stothers}
$abc$-theorem for function fields of characteristic $0$ (see also
\cite{SilABC}).  It would be interesting to treat the question in
positive characteristic as well, but it appears that new ideas or
techniques may be required.  

	We have explained why the conditions $m\notin Y(\varphi,\alpha)$ and 
	$n\notin X(\varphi)$ are necessary. Hence the Ingram-Silverman-Faber-Granville 
	conjecture (which was originally made over number fields) 
	should be modified accordingly.  One may also adapt our proof in this
	paper to resolve their conjecture \textit{assuming the
          $abc$-conjecture} in the context of number fields.  We will
        treat this in a future paper.   

        The original question that motivates this paper is the
        ``simultaneous portrait problem'' (see Theorem \ref{arbitrary
          portraits}) over function fields in several parameters; this
        problem has no obvious analog over number fields.

	Since we are excluding places outside $S$,
	our conclusion involving the finiteness of $\cZ(\tau,s)$ is \textit{the best
	one can hope for}. We also note the remarkable \textit{uniformity} obtained here: the set
	$\cZ(\tau,s)$ only depends on $\varphi$, $K$, $s$, and the lower bound $\tau$ on 
	the canonical height
	of $\alpha$ under $\varphi$ rather than depending on $\alpha$.

The following result is an immediate consequence of Theorem~\ref{MN} in the case when $\varphi\in K[z]$ is a polynomial which is totally ramified at no point in $K$; in this case, $X(\varphi)=\emptyset$ and also $Y(\varphi,\alpha)=\emptyset$ for any $\alpha\in K$ (see Remark~\ref{rem:FG vs IS}).

\begin{cor}\label{eg:polynomial no total ramification}
	Let $K$, $\varphi$, $\tau$, $s$, and $S$ be as in Theorem \ref{MN}. Assume that 
	$\varphi\in K[z]$
	is a polynomial that is totally ramified at no point in $K$ (equivalently,
	$\varphi(z)$ is not linearly conjugate over $K$ to a polynomial of the form $z^d+c$). 
  Then there exists a finite set
  $\cZ(\tau,s)\subset \Z_{\ge 0}\times \N$ depending only on $K$, $\varphi$, $\tau$, and $s$ such that
  for every $\alpha\in K$ satisfying $\hhat_\varphi(\alpha)>\tau$ and
  for every
	$(m,n)\in \left(\Z_{\geq 0}\times \N\right) \setminus \cZ(\tau,s)$ there exists
	a place $\fp\notin S$ such that $\alpha$ has portrait $(m,n)$ mod $\fp$.
\end{cor}

%
%

Theorem~\ref{MN} allows us to prove a result (see
Theorem~\ref{arbitrary portraits}) for simultaneous portraits of
complex numbers realized by polynomials in normal form. Also,
Theorem~\ref{MN} allows us to prove a strong uniform result for
realizing all possible portraits by almost any constant starting point
(see Theorem~\ref{thm:projection}). We will state in
Section~\ref{sect:applications} these two results, together with other
applications of our Theorem~\ref{MN}.

We sketch briefly the plan of our paper. We state in  Section~\ref{sect:applications} applications of Theorem~\ref{MN} (see  Theorems~\ref{arbitrary portraits}, \ref{thm:dual arbitrary portraits} and \ref{thm:projection}). In Section~\ref{subsect:canonical height} we introduce the notation and the basic notions used in the paper. We prove Theorem~\ref{MN} in Section~\ref{proof of main theorem} and then we prove its various applications in Section~\ref{proof of second main theorem}. Finally, we conclude our paper by asking several related questions in Section~\ref{sect:future}.



\section{Applications}
\label{sect:applications}

We say that a polynomial $\varphi(z)$ of degree $d$ is in \emph{normal form} if it is monic and its coefficient of $z^{d-1}$ equals $0$. Note that each polynomial $\varphi$ is linearly conjugate to a polynomial in normal form. Therefore, when discussing preperiodicity portraits in the family of polynomials of degree $d$, it makes sense to restrict the analysis to the case of polynomials in normal form.

Let $d\ge 2$ be an integer, let $k$ be an algebraically closed field of characteristic $0$, let $c_1,\dots, c_{d-1}\in k$, and let $(m_i,n_i)\in \Z_{\ge 0}\times \N$ for $i=1,\dots, d-1$. It is natural to ask whether there exists a polynomial $f\in k[z]$ in normal form and  of degree $d$ such that for each $i=1,\dots, d-1$, the point $c_i$ has preperiodicity portrait $(m_i,n_i)$ for the action of  $f(z)$.

Already Theorem~\ref{MN} solves the above question if $d=1$. Indeed, one considers the polynomial $f(z)=z^2+t\in K[z]$, where $K:=k(t)$ and then Theorem~\ref{MN} yields that at the expense of excluding finitely many portraits   (note also that $X(f)=Y(f,c_1)=\emptyset$ by Remark~\ref{rem:FG vs IS}), there exists a place $\fp$ of $K$ such that $c_1$ has preperiodicity portrait $(m_1,n_1)$ for the action of $f(z)$ modulo $\fp$. Reducing $f(z)$ modulo a place of $K$ is equivalent with specializing $t$ to a value in $k$, hence providing an answer to the above question if $d=1$. 

As a matter of notation, by a 
\textit{co-finite set of portraits} we mean a subset of $\Z_{\geq 0}\times \N$
whose complement is finite. Next result answers the above question of ``simultaneous multiportraits'', and it follows from Theorem~\ref{MN} coupled with an easy fact regarding canonical heights of constant points under the action of a non-isotrivial polynomial (see Lemma~\ref{lem:generic portrait}).

\begin{thm}
\label{arbitrary portraits}
Let $k$ be an algebraically closed field of characteristic 0, let $d\ge 2$ be an integer, and let $c_0,\dots, c_{d-2}\in k$ be $d-1$ distinct elements. Then there
exists a co-finite set of portraits $Z^{(0)}$ depending on $k$ and $d$ such that
for each $(m_0,n_0)\in Z^{(0)}$, there exists
a co-finite set of portraits $Z^{(1)}:=Z^{(1)}(c_0,m_0,n_0)$
depending on $k$, $d$, $c_0$, $m_0$, and $n_0$  such that for each
$(m_1,n_1)\in Z^{(1)}$, there exists a co-finite set
of portraits $Z^{(2)}:=Z^{(2)}(c_0,m_0,n_0,c_1,m_1,n_1)$ 
depending on $k$, $d$, $c_0$,..., $n_1$
such that for each $(m_2,n_2)\in Z^{(2)}$, and so on ...,
there exists 
a co-finite set of portraits $Z^{(d-2)}:=Z^{(d-2)}(c_0,...,n_{d-3})$ 
depending on $k$, $d$, $c_0$,..., $n_{d-3}$
such that for each $(m_{d-2},n_{d-2})\in Z^{(d-2)}$
there exist $a_0,\ldots,a_{d-2}\in k$
such that the following holds. For $0\leq i\leq d-2$,
the point $c_i$ has portrait $(m_i,n_i)$
under $z^d+a_{d-2}z^{d-2}+\ldots+a_1z+a_0$.
\end{thm}

We are interested next in the reverse situation from Theorem~\ref{arbitrary portraits}, i.e. given a set of $(d-1)$ \textit{distinct} portraits $(m_i,n_i)$, for which starting points $c_0,\dots, c_{d-2}\in k$ is there possible to find a polynomial $f\in k[z]$ of degree $d$ and in normal form such that the preperiodicity portrait of $c_i$ with respect to the action of $f(z)$ is $(m_i,n_i)$ for each $i=0,\dots, d-2$? In other words, Theorem~\ref{arbitrary portraits} tells us how many  portraits may be missed for a given set of starting points, while the next result gives information on how many tuples of starting points have to be excluded if a certain set of portraits is to be realized by those starting points.

\begin{thm}\label{thm:dual arbitrary portraits}
		Let $k$ be an algebraically closed field of characteristic 0, let $d\geq 2$ 
		be an integer, and let $(m_0,n_0),\ldots,(m_{d-2},n_{d-2})$
		be distinct elements in $\Z_{\geq 0}\times \N$. Then there exists
		a co-finite subset $T^{(0)}$ of $\bP^1(k)$
		depending on $k$ and $d$ such that for each
		$c_0\in T^{(0)}$, there exits a co-finite
		subset $T^{(1)}:=T^{(1)}(m_0,n_0,c_0)$
		of $\bP^1(k)$
		depending on $k$, $d$, $m_0$, $n_0$, and $c_0$
		such that for each $c_1\in T^{(1)}$, there exists
		a co-finite subset $T^{(2)}:=T^{(2)}(m_0,n_0,c_0,m_1,n_1,c_1)$
		depending on $k$, $d$, $m_0,\ldots,c_1$ such that
		for each $c_2\in T^{(2)}$, and so on ..., there exists
		a co-finite subset $T^{(d-2)}:=T^{(d-2)}(m_0,\ldots,c_{d-3})$
		depending on $k$, $d$, $m_0,\ldots,c_{d-3}$
		such that for each $c_{d-2}\in T^{(d-2)}$
		there exists $a_0,\ldots,a_{d-2}\in k$
		such that the following holds. For $0\leq i\leq d-2$, the point
		$c_i$ has portrait $(m_i,n_i)$
		under $z^d+a_{d-2}z^{d-2}+\ldots+a_1z+a_0$.
\end{thm}

Theorem~\ref{thm:dual arbitrary portraits} follows through an argument
similar to the proof of Theorem~\ref{arbitrary portraits} once we
prove a strong uniform result for the set of possible exceptions of
starting points which cannot realize a given portrait, i.e. the
``dual'' statement from Theorem~\ref{MN}. For this ``dual'' result we
require first the definition of \emph{isotrivial rational maps}. 

\begin{defi}
\label{defi:isotrivial}
Let $k$ be an algebraically closed field of characteristic $0$, and let $K$ be a finitely generated function field over $k$ of transcendence degree equal to $1$. Let $\varphi\in K(z)$ of degree $d\geq 2$. We say that $\varphi(z)$
is isotrivial if there is $\sigma\in \Kbar(z)$ of degree $1$ such that
$\sigma^{-1}\circ \varphi\circ \sigma\in k(z)$. 
\end{defi}

Let $k$, $K$, and $\varphi$ be as in Definition~\ref{defi:isotrivial}. We let $W(\varphi)$ be the set of $(m,n)\in \Z_{\geq 0}\times \N$
such that the set of points having portrait
$(m,n)$ with respect to $\varphi$ is either empty or it is a (proper) subset of $\bP^1(k)$. Since $\varphi$ is not isotrivial, there are at most finitely many $x\in \bP^1(k)$ which are preperiodic for $\varphi$ (see \cite{Matt-nonisotrivial}); hence Theorem~\ref{MN} yields that  there are at most finitely many portraits $(m,n)\in W(\varphi)$ such that $n\notin X(\varphi)$.

\begin{thm}
\label{thm:projection}
Let $S$ be a finite set of places of $K$ containing all places of bad reduction. Let $\varphi\in K(z)$ be non-isotrivial. Then there is a finite
set $T(S)\subset \bP^1(k)$ depending only on $K$, $\varphi$, and $S$ satisfying 
the following property: for every $(m,n)\in \left(\Z_{\geq 0}\times \N\right)\setminus W(\varphi)$ and for
every $\alpha\in \bP^1(k)\setminus T(S)$, there is a place $\fp\in \Omega_K\setminus S$
such that $\alpha$ has portrait $(m,n)$ under the action of $\varphi$ modulo $\fp$.
\end{thm}

\begin{remark}\label{rem:dual main}
	From Lemma~\ref{inv} and Corollary~\ref{cor:all bad n}, we have
	that $(m,n)\notin W(\varphi)$ if and only if $n\notin X(\varphi)$ 
	and
	some point of portrait $(m,n)$ is not constant. The first condition $n\notin 
	X(\varphi)$
	is necessary for Theorem \ref{MN} which will be used in the proof of Theorem~\ref{thm:projection}. Now, if all points of portrait $(m,n)$ are contained in $\bP^1(k)$, then for some $\alpha\in \bP^1(k)$ which is not a point with portrait $(m,n)$ for $\varphi$, we cannot find a place $\fp$ of $K$ such that the portrait of $\alpha$ for $\varphi$ modulo $\fp$ is $(m,n)$ because that would mean that $\alpha$ is in the same residue class modulo $\fp$ as another point in $\bP^1(k)$ (which has portrait $(m,n)$ for $\varphi$ globally). 
\end{remark}

We expect Theorem \ref{thm:dual arbitrary portraits} remains
valid without the condition that the given portraits are distinct. In Theorem 
\ref{thm:dual arbitrary portraits}, we note that the co-finite
set $T^{(i)}$ depends on the previously chosen points $c_0,\ldots,c_{i-1}$
\textit{together with} the portraits
$(m_0,n_0)$,..., $(m_{i-1},n_{i-1})$. It is an interesting problem
to relax such dependence on portraits, for which we present the following result
for cubic polynomials in normal form. By a co-countable subset of a set, 
we mean a subset whose complement is countable. 
%
%
\begin{cor}\label{cubic}
  Suppose that $k$ is algebraically closed of characteristic 0.  Then there 
  exist a co-finite subset $U^{(1)}$ of $k$ such that for every
  $c_1\in U^{(1)}$, there exists a co-countable subset $U^{(2)}(c_1)$ of $k$
  depending on $c_1$ such that for every 
  $c_2\in U^{(2)}(c_1)$, the following holds.
  For every pair
  of portraits $(m_1,n_1)$ and $(m_2,n_2)$, there exist $a,b\in k$
  such that for each $i=1,2$, $c_i$ has portrait $(m_i,n_i)$ under $z^3+az+b$. 
  
  As a consequence, when $k$ is uncountable
  there exist uncountably many $(c_1,c_2)\in k^2$
  such that for every pair of portraits $(m_1,n_1)$ and $(m_2,n_2)$, there exist $a,b\in k$
  such that for each $i=1,2$, $c_i$ has portrait $(m_i,n_i)$ under $z^3+az+b$.
\end{cor}
  






\section{Preliminaries}
\label{subsect:canonical height}

\subsection{Good reduction of rational maps}
\label{good reduction for rational maps}

If $\varphi:\bP^1\to\bP^1$
is a morphism defined over $K$,
then (fixing a choice of homogeneous coordinates)
there are relatively prime homogeneous polynomials $F,G\in K[X,Y]$
of the same degree $d=\deg\varphi$ such that
$\varphi([X,Y])=[F(X,Y):G(X,Y)]$.  
In affine coordinates, $\varphi(z)=F(z,1)/G(z,1)\in K(z)$
is a rational function in one variable. 
Note that by our choice of coordinates,
$F$ and $G$ are uniquely defined up to a nonzero constant
multiple.

Let $\fp$ be a place of
$K$ with valuation ring $\fo_\fp$ and residue field $k$. We define as follows the reduction modulo $\fp$ of a point $P\in \bP^1(K)$. We let $x,y\in \fo_v$ not both in the maximal ideal of $\fo_\fp$ such that $P=[x:y]$ and then the reduction of $P$ modulo $\fp$ is defined to be $r_\fp(P):=[\bar{x}:\bar{y}]$, where $\bar{z}\in k$ is the reduction modulo $\fp$ of the element $z\in \fo_\fp$.

Let 
$\varphi:\bP^1\lra\bP^1$ be a morphism over $K$, 
given by $\varphi([X,Y])=[F(X,Y):G(X,Y)]$, where
$F,G\in \fo_\fp[X,Y]$ are relatively prime homogeneous
polynomials of the same degree
such that at least one coefficient
of $F$ or $G$ is a $\fp$-adic unit.
Let $\varphi_\fp :=[F_\fp:G_\fp]$, where
$F_\fp,G_\fp\in k[X,Y]$ are the reductions
of $F$ and $G$ modulo $\fp$.
We say that $\varphi$ has
{\em good reduction} at $\fp$ if
$\varphi_\fp:\bP^1(k)\lra\bP^1(k)$ is a morphism of the same
degree as $\varphi$. Equivalently, $\varphi$ has good reduction at $\fp$ if $\varphi$ extends as a morphism to the fibre of $\bP^1_{{\rm Spec}(\fo_\fp)}$ above $\fp$. For all but
finitely many
places $\fp$ of $K$, the map $\varphi$ has good reduction at $\fp$ (for more details, see the
comprehensive book of Silverman \cite[Chapter 2]{Silverman-book}).

If $\varphi\in K[z]$ is a polynomial, we can give
the following elementary criterion for good reduction:
$\varphi$ has good reduction at $v$ if and only if all coefficients of $\varphi$ are
$v$-adic integers, and its leading coefficient is a $v$-adic unit. For simplicity, we will always use this criterion when we choose a place $v$ of good reduction for a polynomial $\varphi$.


\subsection{Absolute values and heights in function fields}
%
For any finite extension $L/K$ we let $\Omega_L$ be the set of places of $L$.
For $\fq\in \Omega_L$ and $\fp\in \Omega_K$, if $\fq\mid_K=\fp$ then we
write $\fq\mid \fp$ and let $v_\fq$ and $|\cdot|_\fq$ respectively denote
the extension of $v_\fp$ and of $|\cdot|_\fq$ on $L$.
For every $\fq \in \Omega_\fq$, we let $e(\fq)$ be the ramification index for the extension of places  $\fq\mid \fp$ where $\fp=\fq\mid_K$.


For each $x\in \Kbar$ we define its Weil height as
$$h_K(x)=\frac{1}{[K(x):K]}\sum_{\fp\in \Omega_K}\sum_{\substack{\fq\in \Omega_{K(x)}\\ \fq\mid \fp}}  e(\fq)\cdot \log^+|x|_\fq,$$
where always $\log^+(z):=\log\max\{1,z\}$ for any real number $z$. We prefer to use the notation $h_K$ for the Weil height (normalized with respect to $K$) in order to emphasize the dependence on the ground field $K$ for our definition of the height. For example, if $L/K$ is a finite field extension, and $x\in \Kbar$, then $h_L(x)=[L:K]\cdot h_K(x)$. We extend $h_K$ on $\bP^1(\Kbar)$ by
$h_K(\infty)=0$.

Let $x$ and $y$ be distinct elements of $\bP^1(K)$, we have the following
inequality:
\begin{equation}\label{eq:new Liouville}
\#\{\fp\in\Omega_K:\ r_\fp(x)=r_\fp(y)\}\leq 2(h_K(x)+h_K(y))
\end{equation}
To prove this, we assume that $x,y\in K$ since the case $x=\infty$
or $y=\infty$ is easy. The set in the left-hand side of 
(\ref{eq:new Liouville}) is contained in:
$$\{\fp\in\Omega_K:\ |x-y|_\fp<1\}\cup\{\fp\in\Omega_\fp:\ |x|_\fp>1\ \text{and}\ 
|y|_\fp>1\}$$
whose cardinality is bounded above by:
$$h_K(x-y)+h_K(x)+h_K(y)\leq 2(h_K(x)+h_K(y)).$$


\subsection{Canonical heights for rational maps}

If $\varphi\in\Kbar(z)$ is a rational map of degree  $d\ge 2$, then for each point $x\in\bP^1(\Kbar)$, following \cite{Call-Silverman} we define the canonical height of $x$ under the action of $\varphi$ by:
$$\hhat_\varphi(x)=\lim_{n\to\infty}\frac{h_K(\varphi^n(x))}{d^n}.$$

According to \cite{Call-Silverman}, there is a constant $C_\varphi$ depending
only on $K$ and $\varphi$ such that $| h_K(z) - \hhat_\varphi(z)| <
C_\varphi$ for all $z \in \bP^1(\Kbar)$.

\section{Proof of Theorem \ref{MN}}
\label{proof of main theorem}
Throughout this section, $k$ is an algebraically closed field of characteristic $0$, $K$ is a finitely generated function field over $k$ of transcendence degree equal to $1$, and $\varphi\in K(z)$ is a rational function of degree $d>1$.  Throughout this section, unless stated otherwise all constants depend on $K$
and $\varphi$.
If a constant depends on other arguments, our notation will
clearly indicate them. For example, $C_1,B_2,D_3,\ldots$ denote constants depending
on $K$ and $\varphi$ only, while $C_4(\alpha,\beta,\gamma,\ldots)$ denotes
a constant depending on $K,\varphi,\alpha,\beta,\gamma,\ldots$. 

\subsection{A preliminary estimate} 
At each place $\fp$ of $K$, we use the {\it chordal metric} $d_\fp(\cdot,
\cdot)$ defined as 
\[ d_\fp ( [x:y], [a:b] ) = \frac{ |xb - ya|_\fp}{\left(\max (|x|_\fp,
    |y|_\fp) \right) \left(\max (|a|_\fp, |b|_\fp) \right)}. \]
We see then that for any place $\fp$, we have $r_\fp( [x:y]) =
r_\fp( [a:b] )$ if and only if  $d_\fp ( [x:y], [a:b] ) < 1$. 

We will need the following technical result:
\begin{prop}\label{old}
  Let  $\tau,\delta > 0$ be real numbers, let $i\ge 1$ be an integer, let $\alpha, \beta \in K$ such
  that $\hhat_{\varphi}(\alpha) \geq \tau > 0$, and let $F(z)$ be a monic, separable polynomial with
  coefficients in $K$ whose roots $\gamma_j$ satisfy the following conditions:
\begin{enumerate}
\item $\varphi^i(\gamma_j)=\beta$ for each $j$; 
\item each $\gamma_j$ is not periodic; and
\item for each $j$ and for each $\ell=0,\dots, i-1$, $\varphi^\ell(\gamma_j)\ne \beta$.
\end{enumerate}
For each positive integer $n\geq i$, we let $Z_n$ be the set of places $\fp$ of $K$ such that either
  $\varphi$ has bad reduction at $\fp$ or 
  \begin{equation}\label{eq:in prop}
  \max(d_\fp(\varphi^m(\alpha), \beta), 
|F(\varphi^{n-i}(\alpha))|_\fp   ) <1
  \end{equation}
 for some positive integer $m <
n$.  Then there are constants $C_1(\delta, i, \tau)$,
$C_2(\delta, i, \tau)$, and $B(\delta, i, \tau)$ depending only on $i$, $\delta$, $\tau$, and $\varphi$ such that for 
  all positive integers $n > B(\delta, i, \tau)$, we have
\begin{equation}\label{oldeq}
\# Z_n \leq \delta h_K(\varphi^n(\alpha)) +
(C_1(\delta, i, \tau)+ 2n) h_K(\beta) + C_2(\delta, i, \tau)
\end{equation} 
\end{prop}

Informally, the set $Z_n$ consists of all places $\fp$ such that  $\varphi^{n-i}(\alpha)$ is in the same residue class modulo $\fp$ as one of the roots $\gamma_j$ of $F$ (see Remark \ref{rem:Gamma}) 
and $\beta$ is in the same residue class modulo $\fp$ as an iterate $\varphi^m(\alpha)$ with $m<n$. In other words, we are looking at places $\fp$ such that $\varphi^m(\alpha)$, $\varphi^n(\alpha)$, and $\beta$ have the same
reduction modulo $\fp$. The conclusion of Proposition~\ref{old} is that $\# Z_n$ is bounded above by an explicit quantity whose major term (as $n$ grows)
is $\delta h_K(\varphi^n(\alpha))$. 

\begin{remark}\label{rem:Gamma}
	Let the notation be as in Proposition~\ref{old}. Let $L$ be
	the splitting field of $F(z)$ over $K$. Let $\Gamma$ be the set 
	of places $\fp$ of $K$ such that there is a place $\fq\mid\fp$
	of $L$ and a root $\gamma_j$ of $F$ such that $|\gamma_j|_\fq>1$.
	We have: 
	$$\#\Gamma\leq\sum_j h_K(\gamma_j).$$
	Using $\varphi^i(\gamma_j)=\beta$ so that $\deg(F)\leq d^i$ and
	$h_K(\gamma_j)=\frac{1}{d^i}h_K(\beta)+O(1)$, there
	exist constants $C_3(i)$ and $C_4(i)$ such that: 
	 $$\#\Gamma\leq \sum_j h_K(\gamma_j)\leq C_3(i)h_K(\beta)+C_4(i).$$
  For every place $\fp$ of $K$ outside $\Gamma$, for every $x\in K$, the 
  inequality
  $|F(x)|_\fp<1$ is equivalent to
  the assertion that there is a root $\gamma_j$ of $F$ and a place $\fq\mid \fp$
  of $L$ such that $r_\fq(x)=r_\fq(\gamma_j)$.
\end{remark}

\begin{proof}[Proof of Proposition~\ref{old}.]
 Recall that we have
$\hhat_\varphi(\varphi(z)) = d \hhat_\varphi(z)$ for all $z \in K$ and
that there is a constant $C_\varphi$ such that $|h_K(z)
- \hhat_\varphi(z)| < C_\varphi$ for all $z \in K$. The strategy of our proof is 
to divide $Z_n$ into sets denoted $Y_0$, $Y_2$, and $Y_3$ below in which
the inequality \eqref{eq:in prop} holds when $n-m$ is respectively large, small, and moderate. 

Choose $B(\delta, i, \tau)$ such that the inequalities: 
\begin{equation}\label{eq:B first ineq}
1/d^{B(\delta, i, \tau) +i} < \min\{1,\delta\}/8
\end{equation}
and
\begin{equation}\label{eq:B second ineq}
2(n+1)C_\varphi\leq \frac{\delta}{2}\left(d^n\tau-C_\varphi\right)\leq\frac{\delta}{2}h_K(\varphi^n(\alpha))
\end{equation}
hold for every $n>B(\delta,i,\tau)$. We note that the first inequality in (\ref{eq:B second ineq}) is possible since $d^n$ dominates other terms when $n$ grows, and that
the second inequality in (\ref{eq:B second ineq}) is always true since 
$h_K(\varphi^n(\alpha))+C_\varphi\geq \hhat_\varphi(\varphi^n(\alpha))\geq d^n\tau.$
%
%

For any $\alpha \in \bP^1(K)$ and
all $n > B(\delta, i, \tau)$, we have 
\begin{equation}
\begin{split}
& \sum_{\ell=0}^{n-B(\delta, i, \tau)-i} \#\{\fp\colon r_\fp (\varphi^\ell(\alpha)) =
  r_\fp(\beta) \}  \\
 & \leq  
\sum_{\ell=0}^{n-B(\delta, i, \tau)-i} 2\left( h_K(\varphi^\ell(\alpha)) +
  h_K(\beta) \right)\ \ \text{(by (\ref{eq:new Liouville}))}\\
 & \leq 2n (C_\varphi +h_K(\beta)  ) + 2\sum_{\ell=0}^{n-B(\delta, i, \tau)-i}
 \hhat_\varphi(\varphi^\ell(\alpha))  \\
& = 2n (C_\varphi +h_K(\beta)  ) + \frac{2}{d^{B(\delta, i, \tau)+i}}\sum_{r=0}^{n-B(\delta, i, \tau) - i}
 \frac{\hhat_\varphi(\varphi^n(\alpha))}{d^r}  \\
& \leq   \left(\frac{2}{d^{B(\delta, i, \tau) +i}} \sum_{r=0}^\infty\frac{1}{d^r} \right)
\hhat_\varphi(\varphi^n(\alpha)) + 2n (C_\varphi +h_K(\beta))\\
& \leq  \frac{\min\{1,\delta\}}{2}  \hhat_\varphi(\varphi^n(\alpha))  + 2n (C_\varphi +h_K(\beta))\ \ \text{(by (\ref{eq:B first ineq}))}\\
& \leq  \frac{\delta}{2}  h_K(\varphi^n(\alpha)) + 2(n+1) (C_\varphi +h_K(\beta)) \\
& \leq  \delta h_K(\varphi^n(\alpha)) + 2(n+1) h_K(\beta)\ \ \text{(by (\ref{eq:B second ineq}))}
\end{split}
\end{equation}
Thus, if $Y_0$ is the set of primes such that
$d_\fp(\varphi^\ell(\alpha), \beta) < 1$ for some $\ell \leq n - B(\delta, i, \tau)-i$,
then
\begin{equation}\label{Y_0}
\# Y_0  \leq \delta h_K(\varphi^n(\alpha)) + 2(n+1) h_K(\beta).
\end{equation}

Let $Y_1$ be the set of primes of $K$ for which $\varphi$ does not have good
reduction.  Then clearly,
\begin{equation}\label{Y_1}
\# Y_1 \leq C_{5}
\end{equation}
where $C_{5}$ depends only on $K$ and $\varphi$.  

Now, let $L$ be the splitting field for $F(z)$. Since $\varphi^i(\gamma_j)=\beta$, 
we have:
\begin{equation} \label{z}
\hhat_\varphi(\gamma_j) =\frac{1}{d^i} \hhat_\varphi(\beta).
\end{equation}

Let $Y_2$ be the set of primes $\fp$ outside $Y_1$ and the set $\Gamma$
in Remark \ref{rem:Gamma} such that
\[\max\left(d_\fp(\varphi^m(\alpha), \beta), |F(\varphi^{n-i}(\alpha))|_\fp
\right) <1\]
 for $n-i \leq m < n$.  For each such prime we have
$r_\fq(\varphi^{m - (n - i)}(\gamma_j)) = r_\fq(\beta)$ for some root $\gamma_j$ of $F$ and some prime
$\fq$ of $L$ with $\fq \mid \fp$.  From $m - (n-i) < i$, condition~(3) and \eqref{z}, we have
$\varphi^{m - (n - i)}(\gamma_j) \not= \beta$
and $\hhat_\varphi(\varphi^{m-(n-i)}(\gamma_j))\leq \hhat_\varphi(\beta)$. This
latter inequality implies:
\begin{equation}\label{eq:new z}
h_K(\varphi^{m-(n-i)}(\gamma_j))\leq h_K(\beta)+2C_\varphi.
\end{equation}
Using \eqref{eq:new Liouville} and
\eqref{eq:new z} we
have
\begin{equation}\label{Y_2}
\# Y_2 \leq i (4 h_K(\beta) + 4 C_\varphi).
\end{equation}

Let $Y_3$ be the set of primes $\fp$ outside $Y_1$ and the set $\Gamma$ in Remark
\ref{rem:Gamma} such that
\[ \max\left(d_\fp(\varphi^m(\alpha), \beta),
|F(\varphi^{n-i}(\alpha))|_\fp \right) <1 \] for some positive
integer $m$ with $n - i > m > n- i - B(\delta,
i, \tau)$. If $\fp \in Y_3$, then $\varphi^{m}(\alpha) \equiv
\varphi^n(\alpha) \equiv \beta \pmod{\fp},$ so $\beta$ modulo $\fp$ is in a cycle of
period dividing $n-m$.  There is a prime $\fq\mid \fp$ of $L$
and a root  $\gamma_j$ of $F(z)$ such that
$\gamma_j \equiv \varphi^{n-i}(\alpha) \equiv \varphi^{(n-i)-m}(\beta) \pmod{\fp}$, we see
that $\gamma_j$ is in the same cycle modulo $\fq$.  This implies that $\gamma_j$ modulo $\fq$ has period dividing $n - m$. From \eqref{eq:new z} and $n-m<B(\delta,i,\tau)+i$,
we have:
\begin{equation}\label{eq:another for Y_3}
	\begin{split}
	h_K(\varphi^{n-m}(\gamma_j))+h_K(\gamma_j)&\leq \hhat_\varphi(\varphi^{n-m}(\gamma_j))+\hhat_\varphi(\gamma_j)+2C_\varphi\\
	&\leq \left(d^{B(\delta,i,\tau)}+\frac{1}{d^i}\right)\hhat_\varphi(\beta)+2C_\varphi\\
	&\leq \left(d^{B(\delta,i,\tau)}+\frac{1}{d^i}\right)(h_K(\beta)+C_\varphi)+2C_\varphi
	\end{split}
\end{equation}

Note that each $\gamma_j$ has degree at most $d^i$ over $K$ since    $\varphi^i(\gamma_j)=\beta$ for each $j$. Using \eqref{eq:new Liouville} and \eqref{eq:another for Y_3}, we have:
\begin{equation}\label{Y_3}
\# Y_3 \leq B(\delta, i, \tau) d^i \left(2\left(d^{B(\delta,i,\tau)}+\frac{1}{d^i}\right)(h_K(\beta)+C_\varphi)+4C_\varphi \right).
\end{equation}
Since $Z_n$ is contained in $Y_0 \cup Y_1\cup\Gamma \cup Y_2 \cup Y_3$, we see
that \eqref{oldeq} is a consequence of Remark \ref{rem:Gamma}, \eqref{Y_0}, \eqref{Y_1},
\eqref{Y_2}, and \eqref{Y_3}.
\end{proof}

\subsection{A consequence of the $abc$-theorem for function fields}
%
The following result is crucial for the proof of Theorem~\ref{MN} and it is a consequence of the $abc$-theorem for function fields by Mason-Stothers (see Silverman's
formulation \cite{SilABC}). We will treat the number field case in a future work in which a similar result holds assuming the $abc$-conjecture.
\begin{prop}\label{YamaUse}
  Let $K$ be a function field.  Let $e \geq
  3$ be a positive integer.  Then for any monic $f(z) \in K[z]$ of degree
  $e$ without repeated roots, and for any $\gamma\in K$ we have

\begin{equation} \label{change}   \# \{ \text{primes $\fp$ of $K$}  \; | \;  v_\fp (f(\gamma)) > 0 \}
\geq   h_K(\gamma)  - \left(3 e^2 \sum_{i=1}^e h_K(\eta_i)   + 2 g_K\right) 
\end{equation} 
where $\eta_1, \dots, \eta_e$ are the roots of $f$ in $\bK$, and we denote by $g_L$
the genus of any function field $L$.    
\end{prop}
\begin{proof}
The given inequality holds trivially if $\gamma$ is a root of $f(z)$. We now assume $\gamma$ is not a root of $f(z)$. Let $L = K(\eta_1, \dots , \eta_e)$.   We have
\begin{equation}\label{RH}
[L:K]   \# \{ \text{primes $\fp$ of $K$}  \; | \;  v_\fp (f(\gamma)) > 0
\} \\  \geq \# \bigcup_{i=1}^3  \{ \text{primes $\fq$ of $L$}  \; | \;
v_\fq (\gamma - \eta_i) > 0
\}
\end{equation}
Now, when $L$ ramifies over $K$ at
a prime $\fp$, we must have that either some $\eta_i$ has a pole at a
prime lying over $\fp$ for $1\leq i\leq e$, or $\eta_i - \eta_j$ has a zero at a prime
lying over $\fp$ for $1\leq i<j\leq e$.  Thus, the number of primes $\fp$ of $K$ that are ramified in $L/K$
is bounded by 
$$\sum_{i=1}^e h_K(\eta_i)+\sum_{1\leq i< j\leq e} h_K(\eta_i-\eta_j)\leq e^2\sum_{i=1}^e h_K(\eta_i).$$  
We now apply the Riemann-Hurwitz theorem for $L/K$. Note that for each $\fp$ of $K$ where $L/K$ is ramified,
the total ramification contribution of primes of $L$ lying above $\fp$ in
the Riemann-Hurwitz formula is
at most $[L:K]-1$, hence:
 \begin{equation}\label{disc}
2 g_L - 2  \leq \left(e^2\sum_{i=1}^e h_K(\eta_i)+2g_K-2\right)[L:K]. 
 \end{equation}

Now we construct a change of coordinates  $\sigma$ that takes $\eta_1$, $\eta_2$, $\eta_3$ to 0,
1, $\infty$, i.e. 
\[ \sigma(z) = \frac{\eta_2 - \eta_3}{\eta_2 - \eta_1}   \cdot  \frac{z -
  \eta_1}{z - \eta_3}. \]
Then for any $z$, we have 
\begin{equation}\label{sig}
h_K(z) - 4 \sum_{i=1}^3 h_K(\eta_i) \leq
h_K(\sigma(z)) \leq h_K(z) + 4 \sum_{i=1}^3 h_K(\eta_i).
\end{equation}

Let $B$ be the set of primes $\fq$ of $L$ such that $v_{\fq}(\eta_i)\neq 0$ for some $1 \leq i \leq 3$  or $v_{\fp}(\eta_i - \eta_j)\neq 0$ for some $1
\leq i < j \leq 3$. Then we have
\begin{equation}\label{B}
\#B \leq 2 \left(\sum_{i=1}^3 h_L(\eta_i)+\sum_{1\leq i<j\leq 3} h_L(\eta_i-\eta_j)\right)\leq 6[L:K]\sum_{i=1}^3 h_K(\eta_i).
\end{equation}

For all $\fq \not\in B$, we see that $\sigma$ and
$\sigma^{-1}$ are well defined modulo $\fq$.  Thus, for any $\fq$ outside $B$, we have $v_\fq(\gamma-\eta_i)>0$ for some $1\leq i\leq 3$ if and only if
$\sigma(\gamma)$ has the same reduction mod $\fq$ to 0, 1, or $\infty$. Equivalently:
$v_\fq(\sigma(\gamma)) \not= 0$ or $v_\fq(\sigma(\gamma) - 1)> 0$.
Now, applying the $abc$-theorem for function fields \cite{Mason, Stothers} (especially
Silverman's formulation \cite{SilABC}), we have
\begin{equation}\label{abc} 
\begin{split}
\# \{ \text{primes $\fq$ of $L$}  \; | \; & \text{$v_\fq(\sigma(\gamma)) \not=
0$ or $v_\fq(\sigma(\gamma) - 1) > 0$} \}\\  
& \geq h_L (\sigma(\gamma))  -
(2g_L - 2) \\
& = [L:K] h_K(\sigma(\gamma)) - (2g_L - 2),  
\end{split}
\end{equation}
if $\sigma(\gamma)\notin \bP^1(k)$, where we recall that $k$ is the ground field of $K$.
But the inequality (\ref{abc}) holds trivially
for $\sigma(\gamma)\in k\setminus\{0,1\}$ once we replace $2g_L-2$ by $2g_L$. The fact that $\sigma(\gamma)\notin\{0,1,\infty\}$
follows from the assumption that $\gamma$ is not a root of $f$.

Since $e \geq 3$, combining equations \eqref{RH},
\eqref{disc}, \eqref{sig}, \eqref{B},  and
\eqref{abc} gives the desired inequality \eqref{change}.
\end{proof}



\subsection{Proof of Theorem \ref{MN}: small $m$ or small $n$}
Assume the notation in Theorem \ref{MN}, we prove the existence of $\fp$
such that $r_\fp(\alpha)$ has portrait $(m,n)$ for almost all $(m,n)$
where $n\notin X(\varphi)$, $m\notin Y(\varphi,\alpha)$, and either $m$ or $n$ is small.  
Note that the constants
that appear here may depend on the finite set of places $S$. As before, we will always indicate such dependence.  
We will use the
following very simple lemmas repeatedly.  

\begin{lemma}\label{L1}
Let $\fp$ be a prime of good reduction for $\varphi$.  Suppose that
$r_\fp(\gamma_1) \ne r_\fp(\gamma_2)$ but $r_\fp(\varphi(\gamma_1))
= r_\fp (\varphi(\gamma_2))$.  If $\gamma_1$ is periodic for
$\varphi$ modulo $\fp$, then $\gamma_2$ is not periodic for $\varphi$
modulo $\fp$.  
\end{lemma}
\begin{proof}
We write $r_\fp(\varphi^{n_1}(\gamma_1)) = r_\fp(\gamma_1)$ for some $n_1
> 0$.  Suppose that
$\gamma_2$ was also periodic modulo $\fp$; then we can write
$r_\fp(\varphi^{n_2}(\gamma_2)) = r_\fp(\gamma_2)$ for some $n_2 > 0$.  Since
$r_\fp (\varphi(\gamma_1)) = r_\fp (\varphi(\gamma_2))$ we then must have
\[
r_\fp(\gamma_1) = r_\fp(\varphi^{n_1 n_2} (\gamma_1)) =
r_\fp(\varphi^{n_1 n_2} (\gamma_2)) = r_\fp(\gamma_2),\]
a contradiction.  
\end{proof}

The next lemma is immediate since for any finite extension $L/K$, and for any place $\fq$ of $L$ that lies above the place $\fp$ of $K$, two points of $K$ have the same reduction modulo $\fp$ if and only if they have the same reduction modulo $\fq$.
\begin{lemma}\label{L2}
Let $L$ be an algebraic extension of $K$.
Let $\fp$ be a prime of good reduction for $\varphi$ and suppose that
$\alpha$ has portrait $(m,n)$ modulo $\fq$ for some $\fq \mid \fp$.  Then
$\alpha$ has portrait  $(m,n)$ modulo $\fp$.  
\end{lemma}


We begin by considering the case where $n$ is small.  First, another
lemma.

\begin{lemma}\label{inv}
Assume that $\varphi$ has a point of minimum period $n$ 
and assume one of the following two conditions:
	\begin{itemize}
		\item [(i)] $\varphi(z)$ is not linearly conjugate to $z^{-2}$; or
		\item [(ii)] $\varphi(z)$ is linearly conjugate to $z^{-2}$ and $n\neq 2$.
	\end{itemize}
  Then there exists a point $\beta$ of minimum period $n$ such
  that $\varphi^{-1}(\beta)$ contains a point that is not periodic.  
 \end{lemma}
\begin{proof}
   Note that if $\varphi^{-1}(\beta)$ contains
  only periodic points then $\varphi^{-1}(\beta)$ is a single
  point and thus $\varphi$ ramifies completely at
  $\varphi^{-1}(\beta)$.  Since $\varphi$ has at most two totally
  ramified points, we see then that if $n \geq 3$, then $\varphi$
  ramifies completely over at most two points in any cycle of size
  $n$, and we are done.  If $n = 1$, then after change of coordinates,
  $\varphi$ is a polynomial, which we call it $f$.  Then the fixed points of $f$ are
  solutions to $f(x)  - x = 0$, which has at least one solution (note that $\deg(f)=\deg(\varphi)=d>1$).  If $f$
  is totally ramified at one of these solutions then $f$ is conjugate
  to $f(z) = z^d$, and the fixed point 1 has a non-periodic  image $\xi_d$ of $1$, where $\xi_d$ is a primitive $d$-th root of unity, and we are done.
  Similarly, if $n = 2$ and we have a point $\gamma$ of period 2 such
  that $\varphi$ ramifies completely at both $\gamma$ and
  $\varphi(\gamma)$ then $\varphi(z)$ is conjugate to $z^{-d}$. The given condition implies that $d>2$.  Since $\varphi^2(x) - x =  x^{d^2} - x$, we
  see that $\varphi$ has exactly $d^2 - d > 2$ points of minimum period
  2, so $\varphi$ cannot ramify completely over all of them, so at
  least one of them has non-periodic inverse.  
\end{proof}

Lemma \ref{inv} together with the Kisaka's classification \cite{Kisaka} gives the complete description of $X(\varphi)$:
\begin{cor}\label{cor:all bad n}
	One of the following holds:
	\begin{itemize}
		\item [(i)] $X(\varphi)=\{n\}$ for some $n\in\N$, and moreover, $\varphi$ 
		does not have a point of minimum period $n$. 
		\item [(ii)] $\varphi(z)$ is linearly conjugate to $z^{-2}$ and $ X(\varphi)=\{2\}$.
		\item [(iii)] $X(\varphi)=\emptyset$.
	\end{itemize}
\end{cor}


Our next result yields the conclusion of Theorem~\ref{MN} when the period $n$ is small: 

\begin{prop}\label{prop:small n}
Let $\tau$, $s$, and $S$ be as in Theorem \ref{MN}. Fix a positive integer $n\notin X(\varphi)$. 
Then
there is a constant $M(n,\tau,s)$ depending on $K$, $\varphi$, $n$, $S$, and $\tau$ such that for any $\alpha \in K$ with
$\hhat_\varphi(\alpha) > \tau$ and any $m
> M(n,\tau,s)$, there is a prime $\fp\notin S$
such that $\alpha$ has portrait $(m,n)$ under the action of $\varphi$ modulo $\fp$.  
\end{prop}

\begin{proof}
   By Lemma \ref{inv} and Corollary \ref{cor:all bad n}, there is $\beta\in \bP^1(\Kbar)$
   of minimum period $n$ and 
  non-periodic $\zeta \in \bP^1(\bK)$ such that $\varphi(\zeta) = \beta$.   
  Let $L=K(\zeta)$ and note that $\beta \in
  \bP^1(L)$. We will occasionally apply previous results for $L$ in place of $K$.
  Since $\beta$, $\zeta$, and $L$ depend on $K$, $\varphi$, and $n$, constants depending on
  $L$ and $\varphi$ will ultimately depend on $K$, $\varphi$, and $n$. Let $S_L$ be
  the places of $L$ lying above those in $S$.
 
We see that $\varphi^{-4}(\zeta)$
contains at least four points (see \cite[pp.~142]{Silverman-book}), none of which are periodic.  Thus, 
there is a monic separable polynomial $F$ of degree greater
  than 2 with coefficients in $L$ such such that for every root
  $\gamma$ of $F$, we have $\varphi^4(\gamma) = \zeta$ and $\gamma$ is not periodic. Because $\zeta$ is not periodic, then $\varphi^\ell(\gamma)\ne \zeta$ for each root $\gamma$ of $F$, and for each $\ell=0,\dots, 3$. 

  Let $\alpha \in \bP^1(K)$ such that $\hhat_\varphi(\alpha)>\tau$. There is a constant $C_6(\tau)\geq 4$ such that
  $\varphi^{m-4}(\alpha)\neq \infty$ for $m>C_6(\tau)$ since we can
  simply require $d^{C_6(\tau)}\tau>\hhat_\varphi(\infty)$.  By Proposition~\ref{YamaUse}, there
  is a constant $C_7(n)$ such that:
\begin{equation} \label{m1} \# \{\fq\in\Omega_L:\ v_\fq (F(\varphi^{m-4}(\alpha))) > 0\}
\geq    h_L(\varphi^{m-4}(\alpha))  - C_{7}(n) 
\end{equation}
for every $m>C_6(\tau)$.  

Let $\beta_j=\varphi^j(\beta)$ for $0\leq j\leq n-1$ be elements in the periodic cycle containing 
$\beta$. Let $\cE$ be the set of primes $\fq\in\Omega_L$ of good reduction satisfying
one of the following two conditions:
\begin{equation}\label{eq:condition 1}
	\text{There are $0\leq i<j\leq n-1$ such that $r_\fq(\beta_i)=r_\fq(\beta_j)$.}
\end{equation}

\begin{equation}\label{eq:condition 2}
	\text{$r_\fq(\zeta)=r_\fq(\beta_{n-1})$}
\end{equation}

By \eqref{eq:new Liouville}, there is a constant $C_8(n)$ such that $\#\cE\leq C_8(n)$. This
last inequality together with \eqref{m1} and Remark \ref{rem:Gamma} imply
the existence of a constant $C_9(n,\tau,s)$ such that for every $m>C_9(n,\tau,S)$ the
following holds.
There exists a place $\fq\in\Omega_L\setminus\left(\cE\cup S_L\right)$ such that $\varphi^{m-4}(\alpha)$ and some root $\gamma$ of $F$ have the same reduction modulo a
place of $L(\gamma)$ lying above $\fq$.
Therefore $\varphi^{m}(\alpha)$ and $\zeta$ have the same reduction modulo $\fq$. Since
$\fq\notin \cE$, conditions \eqref{eq:condition 1} and \eqref{eq:condition 2}
together with Lemma~\ref{L1} imply that $\varphi^{m+1}(\alpha)\equiv \beta$ 
is periodic of minimum period $n$ and that $\varphi^m(\alpha)\equiv\zeta$ is 
not periodic modulo $\fq$. Therefore $\alpha$ has portrait $(m+1,n)$ modulo $\fq$. Lemma~\ref{L2} finishes our proof.
\end{proof}

We now consider small values
of $m$:
\begin{prop}
\label{small m}
Let $\tau$, $s$, and $S$ be as in Theorem \ref{MN}. Let $m\ge 0$ be an integer.  
\begin{enumerate}
\item If $m = 0$, then there is a constant $N(\tau,s)$ such that for any
$n > N(\tau)$ and any $\alpha$ with $\hhat_\varphi(\alpha) \geq \tau $, there is a prime $\fp \notin S$ such that
$\alpha$ has portrait $(m,n)$ modulo $\fp$.
\item If $m > 0$, then there is a constant $N(m,\tau,s)$ such that for any
$n > N(m,\tau,s)$ and any $\alpha$ satisfying $\hhat_\varphi(\alpha) \geq \tau $ and $\varphi$ is not totally ramified
at $\varphi^{m-1}(\alpha)$, there is a prime $\fp \notin S$ such that
$\alpha$ has portrait $(m,n)$ under the action of $\varphi$ modulo $\fp$.
\end{enumerate}

\end{prop}
\begin{proof}
We may assume that $\varphi^m(\alpha)\neq \infty$. Otherwise, we can make
the change of variables $z\mapsto \frac{1}{z}$. 

If $m = 0$, then let $F$ be a monic separable polynomial of
degree greater than two such that every root $\gamma$ of $F$
satisfies $\varphi^5(\gamma) = \alpha$. Since $\alpha$ is not preperiodic (because it has positive canonical height), then also each root $\gamma$ of $F$ is not periodic, and moreover, $F^\ell(\gamma)\ne \alpha$ for $\ell=0,\dots, 4$. 
If $m > 0$,  and $\varphi$ is not totally ramified
at $\varphi^{m-1}(\alpha)$, then let $F$ be a monic separable polynomial of
degree greater than two such that each $\gamma$ of $F$
satisfies: 
\begin{enumerate}
\item[(i)] $\varphi^5(\gamma) = \varphi^m(\alpha)$ and  $\varphi^\ell(\gamma) \not=\varphi^m(\alpha)$ for $\ell = 0, \dots, 4$ (also, $\gamma$ is not periodic because $\varphi^m(\alpha)$ is not periodic); and  
\item[(ii)] $\varphi^4(\gamma)
\not= \varphi^{m-1}(\alpha)$.
\end{enumerate}  
The fact that $\deg(F)>2$ (or even $\deg(F)\geq 4$) follows from \cite[pp.~142]{Silverman-book}.

As before, there is a constant $C_{10}(\tau)\geq 5$ such that for every $n>C_{10}(\tau)$,  
we have $\varphi^{m+n-5}(\alpha)\neq \infty$. Then, by Proposition~\ref{YamaUse}
there is a constant $C_{11}(m)$ such that:
\begin{equation} \label{n1} \#\{\fp\colon v_\fp (F(\varphi^{m+n-5}(\alpha))) > 0\}
\geq    h_K(\varphi^{m+n-5}(\alpha))  - C_{11}(m) 
\end{equation}
for every $n>C_{10}(\tau)$.   

For $n>C_{10}(\tau)$, we let $W_n$ be the set of primes $\fp$ of $K$ such that either
$\varphi$ has bad reduction at $\fp$ or 
  \[\max(d_\fp(\varphi^{\ell}(\varphi^m(\alpha)),\varphi^m(\alpha)), 
|F(\varphi^{n-5}(\varphi^m(\alpha)))|_\fp   ) <1\] for some positive integer $\ell
< n$.  Proposition~\ref{old} (for $\beta=\varphi^m(\alpha)$, $i=5$, and $\delta=\frac{1}{2d^5}$) shows that there are constants
$C_{12}(m,\tau)>C_{10}(\tau)$ and $C_{13}(m,\tau)$ such that for all $n >C_{12}(m,\tau)$, we have
\begin{equation}\label{n2}
\# W_n  \leq \frac{1}{2d^5} h_K(\varphi^{m+n}(\alpha)) + nC_{13}(m,\tau).
\end{equation}


If $m > 0$, let $\cE$ be the set places $\fp\in\Omega_K$ of good reduction such that
$r_\fq(\varphi^4(\gamma)) = r_\fq(\varphi^{m-1}(\alpha))$ for some
root $\gamma$ of $F$
and some place $\fq$ of $K(\gamma)$ above $\fp$.  If $m = 0$, we let $\cE$ be the empty set.  There is a constant $C_{14}(m,\tau)$ such that:
\begin{equation}\label{n3}
\# \cE  \leq C_{14}(m,\tau).
\end{equation}

Note that $h_K(\varphi^{r}(\alpha))=d^r\hhat_\varphi(\alpha)+O(1)$
for every integer $r\geq 0$ and every $\alpha\in \bP^1(\Kbar)$,
where $O(1)$ only depends on $K$ and $\varphi$. From the right-hand sides of \eqref{n1}, \eqref{n2}, and \eqref{n3} together with Remark \ref{rem:Gamma},
there is a constant $C_{15}(m,\tau,s)>C_{12}(m,\tau)$ such that for every $n>C_{15}(m,\tau,s)$ and every $\alpha\in \bP^1(K)$ satisfying $\hhat_\varphi(\alpha)\geq \tau$,
there is a place $\fp\in \Omega_K$ satisfying the following conditions:
\begin{itemize}
	\item [(I)] $\fp\notin W_n\cup\cE\cup S$.
	\item [(II)] $|F(\varphi^{m+n-5}(\alpha))|_\fp<1$ and there are a root $\gamma$ of $F$ and a place $\fq\mid\fp$
	of $K(\gamma)$ such that $r_\fq(\varphi^{m+n-5}(\alpha))=r_\fq(\gamma)$.
\end{itemize}
Now the condition (I) implies $r_\fp(\varphi^{m+n}(\alpha))=r_\fq(\varphi^m(\alpha))$.
The conditions $\fp\notin W_n$ and $|F(\varphi^{m+n-5}(\alpha))|_\fp<1$ imply
$d_\fp(\varphi^{m+\ell}(\alpha),\varphi^m(\alpha))\geq 1$ for every positive integer $\ell<n$. Hence
$r_\fp(\varphi^m(\alpha))$ has minimum period $n$. Finally when $m>0$, by the
definition of $\cE$ and Lemma \ref{L1}, the condition $\fp\notin\cE$
implies that $\varphi^{m-1}(\alpha)$ is not periodic modulo $\fp$. Therefore $\alpha$
modulo $\fp$ has portrait $(m,n)$.
%
\end{proof}

\subsection{Proof of Theorem \ref{MN}: large $m$ and $n$}
Let $\tau$, $s$, and $S$ be as in Theorem \ref{MN}. We now show that there are
constants $C_{16}(\tau,s)$ and $C_{17}(\tau,s)$ such that for every
$m\geq C_{16}(\tau,s)$, $n\geq C_{17}(\tau,s)$,
and $\alpha\in\bP^1(K)$ satisfying $\hhat_\varphi(\alpha)\geq \tau$,
there is a place $\fp\in\Omega\setminus S$ such that
$\alpha$ has portrait $(m,n)$ modulo $\fp$. Combining this with
Propositions \ref{prop:small n} and \ref{small m}, we  finish
the proof of Theorem \ref{MN}.

There is a constant $C_{18}(\tau)\geq 2$ such that for every $m\geq C_{18}(\tau)$
and every $\alpha\in\bP^1(K)$ satisfying $\hhat_\varphi(\alpha)\geq \tau$,
we have $\varphi^m(\alpha)\neq \infty$ and:
\begin{equation}\label{eq:C18}
  \text{$\varphi^{-2}(\varphi^m(\alpha))$
   contains neither $\infty$ nor any ramification points of $\varphi$.}
\end{equation}
Note that the above conditions are satisfied when $d^{m-2}\tau$ is
greater than the canonical heights of $\infty$ and  of the ramification points of $\varphi$. 
%
%

Fix any $m\geq C_{18}(\tau)$ and $\alpha\in \bP^1(K)$ satisfying $\hhat_\varphi(\alpha)\geq \tau$. Now, there
 are  $d^2 - 1 \geq 3$ distinct points $\eta_i \in \bK$ such that $\eta_i\ne \varphi^{m-2}(\alpha)$ and 
  $\varphi^2(\eta_i) = \varphi^m(\alpha)$ for $i = 1, \dots , d^2 - 1$.
  Let $e = d^2 - 1$ and $F(z) = \prod_{i=1}^e (z - \eta_i)$.  Applying
  Proposition~\ref{YamaUse} and using the fact that $| h_K - \hhat_\varphi|$
  is uniformly bounded, we have that there exist constants $C_{19}$ and $C_{20}$ such that the following inequality holds for every positive integer $n$:
\begin{equation}\label{eq:use Yamanoi}
\begin{split}
\#\{ \fp \colon    v_\fp (F(\varphi^{m+n-2}(\alpha))) > 0
\}  &\geq
\hhat_\varphi(\varphi^{m+n-2}(\alpha))   - C_{19}  \sum_{i=1}^e \hhat_\varphi(\eta_i)  -
C_{20}\\
	& =\hhat(\alpha)\left(d^{m+n-2}-eC_{19}d^{m-2}\right)-C_{20}. 
\end{split} 
\end{equation}
Hence there is a constant $C_{21}(\tau,s)$ such that for every $n>C_{21}(\tau,s)$,
we have:
\begin{equation}\label{eq:3/4}
	\#\{\fp\in\Omega_K\setminus S:\ v_\fp(F(\varphi^{m+n-2}(\alpha)))>0\}\geq
	\frac{3}{4}\hhat_\varphi(\alpha)d^{m+n-2}
\end{equation}

Let $L$ be the splitting field of $K$. We now argue as in Remark \ref{rem:Gamma} as follows. Let $\Gamma$
be the set of places $\fp\in\Omega_K$ such that there are some $\fq\mid\fp$ in
$\Omega_L$ and some root $\eta_i$ of $F$ satisfying $|\eta|_\fq>1$. Then we have:
\begin{equation}\label{eq:size Gamma}
	\#\Gamma\leq\sum_{i=1}^e h_K(\eta_i)\leq e\hhat_\varphi(\alpha)d^{m-2}+O(1)
\end{equation}
where $O(1)$ only depends on $K$ and $\varphi$.
As in Remark \ref{rem:Gamma}, if $\fp\notin \Gamma$, $\varphi$ has good
reduction at $\fp$, and if $|F(\varphi^{m+n-2}(\alpha))|_\fp<1$
then there is a root $\eta_i$ of $F$, and a place $\fq$ of $L$ lying above $\fp$  such that
$r_\fq(\eta_i)=r_\fq(\varphi^{m+n-2}(\alpha))$. This implies
$r_\fp(\varphi^m(\alpha))=r_\fp(\varphi^{m+n}(\alpha))$.
Therefore, from \eqref{eq:3/4} and \eqref{eq:size Gamma} there
is a constant $C_{22}(\tau,s)>C_{21}(\tau,s)$ such that for every
$n>C_{22}(\tau,s)$, we have:
\begin{equation}\label{eq:main}
	\#\{\fp\in\Omega_K\setminus S:\ r_\fp(\varphi^{m+n}(\alpha))=r_\fp(\varphi^m(\alpha))\}\geq \frac{1}{2}\hhat_\varphi(\alpha)d^{m+n-2}
\end{equation}
 
%

Let $\cE_1$ be the set of
primes $\fp\in\Omega_K$ such that there are a prime $\fq\mid\fp$ of $L$ and
a root $\eta_i$ of $F$ satisfying 
$r_\fq(\varphi(\eta_i)) = r_\fq(\varphi^{m-1}(\alpha))$. 

If $r_\fq(\varphi(\eta_i)) = r_\fq(\varphi^{m-1}(\alpha))$ then we have either 
$v_\fq(\varphi^{m-1}(\alpha))<0$
or $v_\fq(\varphi(\eta_i)-\varphi^{m-1}(\alpha))>0$. Therefore: 
  \begin{equation}\label{eq:E1}
  \begin{split}
  \#\cE_1&\leq h_K(\varphi^{m-1}(\alpha))+\sum_{i=1}^e h_K(\varphi(\eta_i)-\varphi^{m-1}(\alpha))\\
  &\leq h_K(\varphi^{m-1}(\alpha))+\sum_{i=1}^e (h_K(\varphi(\eta_i))+h_K(\varphi^{m-1}(\alpha)))\\
  &=(e+1)\hhat_\varphi(\alpha)d^{m-1}+e\hhat_\varphi(\alpha)d^{m-2}+O(1)
  \end{split}
  \end{equation}
where $O(1)$ depends only on $K$ and $\varphi$.

Hence there is $C_{23}(\tau,s)$ such that
for every $n>C_{23}(\tau,s)$,we have:
\begin{equation}\label{eq:size E1}
	\#\cE_1\leq \frac{1}{8}\hhat_\varphi(\alpha)d^{m+n-2}.
\end{equation}



Let $\cE_2$ be the set of primes $\fp$ of good reduction such that $r_\fp (\varphi^{n'}
(\varphi^m(\alpha))) = r_\fp(\varphi^m(\alpha))$ for $n'$ a proper
divisor of $n$. As before, we either have $v_{\fp}(\varphi^m(\alpha))<0$
or $v_\fp(\varphi^{m+n'}(\alpha)-\varphi^m(\alpha))>0$. Hence, we have
 \begin{equation}\label{eq:E2}
 \# \cE_2 \leq h_K(\varphi^m(\alpha))+\sum_{p | n} h_K\left( \varphi^{n/p}(\varphi^m(\alpha)) -
  \varphi^m(\alpha) \right)
  \end{equation}
where $p$ ranges over the distinct prime factors of $n$.
There is an absolute constant $C_{24}$ such that for all  $n \geq C_{24}$, the number of distinct prime
factors of $n$ is less than $\log n$.  Then
for $n\geq C_{24}$:  
\begin{equation} \begin{split} 
 \# \cE_2 & \leq  \hhat_\varphi(\varphi^m(\alpha))+ \sum_{p | n}\left(\hhat_\varphi\left( \varphi^{n/p}(\varphi^m(\alpha))\right) +
 \hhat_\varphi(\varphi^m(\alpha))\right)+C_{25}\log(n)\\
 & \leq  (\log n+1)\hhat_\varphi(\alpha)d^m+  (\log n)\hhat_\varphi(\alpha)d^{m+\frac{n}{2}}+C_{25}\log n 
 \end{split}\end{equation}
where $C_{25}$ depends only on $K$ and $\varphi$. 

Since $d^{n}$ dominates both $(\log(n))d^{n/2}$
and $(\log n+1)$ when $n$ grows sufficiently large (and independently from $m$), there exists a constant $C_{26}>C_{24}$ depending only on $K$ and $\varphi$ such that
for $n>C_{26}$, we have:
\begin{equation}\label{eq:size E2}
	\#\cE_2\leq \frac{1}{8}\hhat_\varphi(\alpha) d^{m+n-2}
\end{equation}

For $m>C_{18}(\tau)$ and for  $n>\max\left\{C_{22}(\tau,s),C_{23}(\tau,s),C_{26}\right\}$, from
\eqref{eq:main}, \eqref{eq:size E1}, and \eqref{eq:size E2} there exist
at least $\frac{1}{4}\hhat_\varphi(\alpha)d^{m+n-2}$
 many primes $\fp\in \Omega_K\setminus S$
such that the following conditions hold:
\begin{itemize}
	\item [(I)] $r_\fp(\varphi^{m+n}(\alpha))=r_\fp(\varphi^m(\alpha))$.
	\item [(II)] $\fp\notin \cE_1\cup\cE_2$.
\end{itemize}

Condition (I) together with $\fp\notin \cE_2$ imply that $\varphi^m(\alpha)$
has minimum period $n$ under the action of $\varphi$ modulo $\fp$. Condition $\fp\notin \cE_1$
together with Lemma \ref{L1} imply that
$\varphi^{m-1}(\alpha)$ is not periodic modulo $\fp$. Hence $\alpha$
has portrait $(m,n)$ modulo $\fp$, which finishes the proof of Theorem \ref{MN}.

\section{Proof of the applications of Theorem~\ref{MN}}
\label{proof of second main theorem}

Using Theorem~\ref{MN} we can prove now its applications. First we prove Theorem~\ref{arbitrary portraits}, and then we will prove Theorems~\ref{thm:projection} and \ref{thm:dual arbitrary portraits}.


\subsection{Simultaneous multiple portraits}

We begin with a few simple lemmas.
\begin{lemma}\label{lem:generic portrait}
Let $\varphi(z)\in K(z)$ be a rational function of degree $d>1$ and let $\alpha\in \bP^1(K)$ be a preperiodic point with portrait $(m,n)$. Then for all but finitely many places $\fp\in\Omega_K$
of good reduction, $\alpha$ modulo $\fp$ has portrait $(m,n)$.
\end{lemma}
\begin{proof}
	Let $\cE$ be the set of places $\fp\in\Omega_K$ of good reduction
	such that the following two conditions hold:
	\begin{itemize}
		\item [(i)] If $m>0$, we have $r_\fp(\varphi^{m-1}(\alpha))=r_\fp(\varphi^{m+n-1}(\alpha))$.
		\item [(ii)] For some prime divisor $\ell$ of $n$, we have
		$r_\fp(\varphi^{m+\frac{n}{\ell}}(\alpha))=r_\fp(\varphi^m(\alpha))$. 
	\end{itemize}
	By \eqref{eq:new Liouville}, $\cE$ is finite. By Lemma \ref{L1}, for
	every place $\fp\in \Omega_K\setminus \cE$, we have $\alpha$ modulo $\fp$
	has portrait $(m,n)$.
\end{proof}

We have the following lemma for determining when polynomials in
normal form are isotrivial.

\begin{lemma}\label{lem:iso}
Let $k$ be an algebraically closed field of characteristic $0$, and
let $K$ be a finitely generated function field over $k$ of
transcendence degree equal to $1$.  Let $\varphi(z) = z^d + a_{d-2}
z^{d-2} + \dots + a_0 \in K[z]$ where $d \geq 2$.  Then $\varphi$ is isotrivial if and
only if $\varphi \in k[z]$.  
\end{lemma}
\begin{proof}
Suppose that $\sigma^{-1}\circ \varphi\circ \sigma\in k(z)$.  Let
${\tilde \varphi}$ denote $\sigma^{-1}\circ \varphi\circ \sigma$.
Then $h_{\varphi}(x) = h_{\tilde \varphi}(\sigma^{-1}(x)) = h(\sigma^{-1}(x))$ for all
$x$.  Let
$\beta$ denote the point at infinity.   Then $h(\sigma^{-1}(\beta)) =
0$, so $\beta \in \bP^1(k)$.  Hence, after composing $\sigma$ with a
degree one element of $k(z)$, we may suppose that $\sigma(\beta) =
\beta$, which means that  $\sigma$ is a polynomial $b_1 z + b_0 \in
\Kbar[z]$.  Since   
\[ \sigma^{-1}\circ \varphi\circ \sigma\in k(z) =
b_1^{d-1} z^d + d b_1^{d - 2} b_0 z^{d-1}
+ \text{lower order terms},\]
we see that $b_1 \in k$ and that $b_0$ must therefore be in $k$ as well.  Thus,
$\sigma \in k[z]$, so $\varphi \in k[z]$.  
\end{proof}

The following lemma is crucial for the proof of Theorem~\ref{arbitrary portraits}.

\begin{lemma}
\label{preperiodicity lemma}
Let $k$ be an algebraically closed field of characteristic $0$, let $K$ be a finitely generated function field over $k$ of transcendence degree equal to $1$, let $d \geq 2$ be an integer and let $m$ be an integer such that $0
\leq m \leq d-2$.  Let 
\[ f(z) = z^d + a_{d-2} z^{d-2} + \dots + a_0 \]
where 
\begin{enumerate}
\item $a_i \in K$ for all $i$;
\item $a_i \in k$ for $i > m$;  and 
\item there is some $j \leq m$ such that $a_j \in K \setminus k$.
\end{enumerate}
Then there are at most $m$ distinct constants $x \in k$ such that
$\hhat_f(x) < \frac{1}{d}$. 
\end{lemma}

\begin{proof}
By (3), there is some place $\fp$ of $K$ such that $|a_j|_\fp >1$ for some
$j \leq m$; fix this $\fp$.  Take
any $x \in K$ such that $|x|_\fp \geq \max_i |a_i|_\fp$.  Then, for all $i
\leq d -2$, we have $|a_i x^i|_\fp \leq |x^{i+1}|_\fp < |x^d|_\fp$, so 
$|f(x)|_\fp = |x|_\fp^d$.  By induction, we then have $|f^n(x)|_\fp = |x|_\fp^{d^n}$ for all $n$.  Thus, in particular for any $\alpha \in k$ such that $|f(\alpha)|_\fp
\geq \max_i |a_i|_\fp$, we have  
\[ \hhat_f(\alpha) = \frac{1}{d} \hhat_f (f(\alpha)) \geq \frac{1}{d}. \] 

Thus, it suffices to show that there are at most $m$ constants
$x \in k$ such that  $|f(x)|_\fp < \max_i |a_i|_\fp$. Let $N =-
\min_i v_\fp(a_i)$; then $N>0$.  Let $\pi\in K$ be a generator for the maximal ideal $\fp$.  Then it suffices to show that at there are at most $m$ constants
$x \in k$ such that  $|\pi^N f(x)|_\fp < 1$.  Now, for each $a_i$, we
have that $\pi^N a_i$ is in the local ring at $\fp$.  We let $b_i$
denote the image of $\pi^N a_i$ in the residue field $k_\fp$ of $\fp$ which
is canonically isomorphic to $k$.
If $|\pi^N f(x)|_\fp <1$
for $x \in k$, then we have
\begin{equation}\label{b}
b_m x^m + \dots + b_0 = 0
\end{equation}
since $b_i = 0$ for all $i > m$ by (2) (note that $N\ge 1$).  Because $b_i \not= 0$ for some
$i \leq m$, we see that \eqref{b} has at most $m$ solutions $x$, and our
proof is complete. 
\end{proof}

\begin{proof}[Proof of Theorem~\ref{arbitrary portraits}.]
When $d=2$, we have $f(z)=z^2+a_0$ defined over the function field $k(a_0)$. By Corollary \ref{cor:all bad n} and the Kisaka's classification \cite{Kisaka},
we have that $X(f)=\emptyset$. For every constant $c_0\in k$, there does
not exist a \textit{positive} integer $m$ such that $f^m(c_0)=0$, hence
the set $Y(f,c_0)$ is empty. By Lemma \ref{preperiodicity lemma}, $\hhat_f(c_0)\geq 1/2$. We apply Theorem \ref{MN} to get the desired conclusion.

From now on, assume $d\geq 3$. Every polynomial of degree $d$ in normal form whose coefficient
$a_{d-2}$ is nonzero is not totally ramified at any point (other than $\infty$). We
will repeatedly use this observation and apply Theorem \ref{MN}
(see also Corollary~\ref{eg:polynomial no total ramification}). 
We prove Theorem~\ref{arbitrary portraits} by showing inductively the existence of the $a_i$'s realizing the portraits for the $c_i$'s. The $a_i$'s will be first independent variables, and then we make a series of specializations of the $a_i$'s which  we  call in turn  $a_{i,0}, a_{i,1}, \cdots$, until we specialize all the variables to values in $k$.

First, we let $k_0:=\overline{k(a_1,a_2,\dots,a_{d-2})}$, and we apply Theorem~\ref{MN} (see
Corollary~\ref{eg:polynomial no total ramification}) to the polynomial 
$$f_0(z)=z^{d}+a_{d-2}z^{d-2}+\cdots + a_1z+a_0$$ 
defined over the function field $K_0:=k_0(a_0)$ with the starting point $c_0$. We note that by Lemma~\ref{preperiodicity lemma}, we know that $\hhat_{f_0}(c_0)\ge \frac{1}{d}$. Take the exceptional set of of places $S_0$
to be the set containing only the place ``at infinity'' of $K_0=k_0(a_0)$
which is the only pole of $a_0$.
Hence we obtain the existence of a co-finite set $Z^{(0)}\subset \Z_{\ge 0}\times \N$ of portraits such that for all $(m_0,n_0)\in Z^{(0)}$, there exists $a_{0,0}\in k_0$ such that $c_0$ has portrait $(m_0,n_0)$ with respect to $$f_1(z):=z^{d}+a_{d-2}z^{d-2}+\cdots +a_2z^2+a_{1}z+a_{0,0}.$$ 
This is the polynomial $f_0$ obtained after specializing $a_0$ to $a_{0,0}\in k_0$.  This specialization is equivalent with reducing $f_0$ modulo a place of $K_0$. Fix
$(m_0,n_0)\in Z^{(0)}$ and a corresponding $a_{0,0}$.


Next we let $k_1:=\overline{k(a_2,\dots, a_{d-2})}$ and we regard  
$f_1(z)$ above
as a polynomial defined over the function field $K_1:=k_1(a_1,a_{0,0})$ (note that ${\rm trdeg}_{k_1}K_1=1$ because $a_{0,0}\in \overline{k_1(a_1)}$). Applying Lemma~\ref{preperiodicity lemma} to $f_1(z)$ we conclude that there exists at most one constant point $c\in k_1$ such that $\hhat_{f_1}(c)<\frac{1}{d}$.  Since, by construction, $c_0\in k\subset k_1$ is preperiodic of portrait $(m_0,n_0)$ for $f_1$, we must
have $\hhat_{f_1}(c_1)\geq \frac{1}{d}$.
Let the exceptional set of places $S_1$ consist of places $\fp$ of $K_1$ where $a_1$ or $a_{0,0}$
has a pole, or when $c_0$ does \textit{not} have portrait 
$(m_0,n_0)$ modulo $\fp$. The set $S_1$ is finite by Lemma~\ref{lem:generic portrait}.
Therefore we can apply Theorem~\ref{MN} (see Corollary~\ref{eg:polynomial no total ramification}) and obtain a co-finite set $Z^{(1)}\subset \Z_{\ge 0}\times \N$ of portraits such that
for each $(m_1,n_1)\in Z^{(1)}$, there exists $a_{1,1}\in k_1$ (and in turn $a_{0,1}\in k_1$) such that $c_0$ has portrait $(m_0,n_0)$ and $c_1$ has portrait $(m_1,n_1)$ under the action of  
$$f_2(z):=z^{d}+a_{d-2}z^{d-2}+\cdots + a_2z^2 +a_{1,1}z+a_{0,1}.$$ 
This comes from reducing $a_1$ and $a_{0,0}$ modulo a place in $K_1$ outside $S_1$.
Fix $(m_1,n_1)\in Z^{(1)}$ and also fix corresponding $a_{1,1}$ and $a_{0,1}$.


The above process could be done inductively as follows. Let $i\in\{1,\dots, d-2\}$,  and assume we previously found $Z^{(0)}$, $Z^{(1)}$,..., $Z^{(i-1)}$ and 
fixed $(m_j,n_j)\in Z^{(j)}$ for $0\leq j\leq i-1$ together with
$a_{0,i-1},a_{1,i-1},\ldots,a_{i-1,i-1}\in k_{i-1}$
where $k_{i-1}:=\overline{k(a_i,\ldots,a_{d-2})}$ such that the following hold.
Let $k_i:=\overline{k(a_{i+1},\ldots,a_{d-2})}$
with the understanding that $k_i=k$ when $i=d-2$. Let
$$f_i(z):=z^d+a_{d-2}z^{d-2}+\ldots+a_iz^i+a_{i-1,i-1}z^{i-1}+\ldots+a_{1,i-1}z+a_{0,i-1},$$ 
which is a polynomial defined over the function field $K_i:=k_i(a_i,a_{i-1,i-1},\ldots,a_{0,i-1})$. We now have
that for $0\leq j\leq i-1$, the point $c_j$ has portrait $(m_j,n_j)$
under the action of $f_i$. 

Lemma \ref{preperiodicity lemma} now 
asserts that there are at most $i$ constants $c\in k_i$
such that $\hhat_{f_i}(c)<\frac{1}{d}$. Since $c_0,\ldots,c_{i-1}$
are such constants, we must have $\hhat_{f_i}(c_i)\geq \frac{1}{d}$. Let $S_i$
be the set of places $\fp$ of $K_i$ such that 
$a_i$ has a pole, or for some $0\leq j\leq i-1$ the element $a_{j,i-1}$
has a pole, or for some $0\leq j\leq i-1$ the point $c_j$
modulo $\fp$ does not have portrait $(m_j,n_j)$. This set $S_i$
is finite by Lemma \ref{lem:generic portrait}. By Theorem \ref{MN} (see also
Corollary~\ref{eg:polynomial no total ramification}), 
there exists a co-finite set $Z^{(i)}\subset \Z_{\ge 0}\times \N$ of portraits  such that for every $(m_i,n_i)\in Z^{(i)}$
there exist $a_{0,i},\ldots,a_{i,i}\in k_i$ satisfying the following. For $0\leq j\leq i$, the point $c_j$ has portrait $(m_j,n_j)$
under:
$$f_{i+1}(z):=z^{d}+a_{d-2}z^{d-2}+\ldots+a_{i+1}z^{i+1}+a_{i,i}z^i+\ldots+a_{1,i}z+a_{0,i}.$$
We continue the above process until $i=d-2$, which finishes the proof of Theorem~\ref{arbitrary portraits}.
\end{proof}


\subsection{Almost any portrait is realized by almost any starting point}
Let $\varphi(z)\in K(z)$ having degree $d\geq 2$. In the proof of Theorem~\ref{thm:projection} we will use the following easy fact.
\begin{lemma}\label{lem:finite}
	Let $\varphi(z)\in K(z)$ be non-isotrivial. 
		Then the set $\{\alpha\in \bP^1(K):\ Y(\varphi,\alpha)\neq \emptyset\}$
		is finite.
\end{lemma}
\begin{proof}
  We use the following two properties following
  from the fact that $\varphi$ is non-isotrivial (see \cite{Matt-nonisotrivial}):
  \begin{itemize}
  	\item [(i)] The set $\Prep_{\varphi}(K)$
  	of preperiodic points in $\bP^1(K)$ is finite.
  	\item [(ii)] There is a positive lower bound
  	$\tau$ for the canonical height of points in
  	$\bP^1(K)\setminus \Prep_{\varphi}(K)$.
  \end{itemize}
  
	Let $R$ be
	the finite set (possibly empty) of points in $\bP^1(\Kbar)$
	where $\varphi$ is totally ramified at.
  There exists $M$ such that
  for every $m>M$ and every 
  $\alpha\in\bP^1(K)\setminus \Prep_{\varphi}(K)$, $\varphi$ is not totally ramified
  at $\varphi^m(\alpha)$. To see this, we 
  simply require that $\hhat_\varphi(\varphi^m(\alpha))>d^M\tau$ is greater than the canonical height
  of any point in $R$. Then the given set in the lemma
  is contained in the finite set:
  $$\Prep_{\varphi}(K)\cup R\cup\varphi^{-1}(R)\ldots
  \cup\left(\varphi^{M}\right)^{-1}(R).$$	
\end{proof}

\begin{proof}[Proof of Theorem~\ref{thm:projection}.]
	Let $T_1$ be the finite set of points in $\bP^1(k)$ consisting of either
	preperiodic points or the points in the set in Lemma \ref{lem:finite}.
	We now have $Y(\varphi,\alpha)=\emptyset$
	for every $\alpha\in \bP^1(k)\setminus T_1$. Note that if $(m,n)\notin W(\varphi)$
	then
	 $n\notin X(\varphi)$ (see Remark \ref{rem:dual main}).

	There is a lower bound $\tau$ on the canonical heights of points in 
	$\bP^1(k)\setminus T_1$ (see \cite{Matt-nonisotrivial}). By Theorem \ref{MN},
	there is a finite set $\cZ(\tau,|S|)$
	such that for every
	$(m,n)\in \left(\Z_{\geq 0}\times \N\right)\setminus (\cZ(\tau,|S|)\cup W(\varphi))$
	and for every $\alpha\in \bP^1(k)\setminus T_1$
	there exists a place $\fp\notin S$
	such that $\alpha$ has portrait $(m,n)$ modulo $\fp$.
	
	Hence it suffices to \textit{fix} an $(m,n)\in \cZ(\tau,|S|)\setminus W(\varphi)$
	and prove that there exists 
	a finite subset $T_2$ of $\bP^1(k)$ (possibly depending on 
	$K$, $\varphi$, $(m,n)$, and $S$) such that the following holds. For every
	$\alpha\in \bP^1(k)\setminus T_2$, there exists a place
	$\fp\notin S$ such that $\alpha$ has portrait $(m,n)$ 
	modulo $\fp$.
	
	Now let $C$ be a nonsingular projective curve over $k$ whose function field is $K$.
	We identify places of $K$ with points in $C$. Choose a Zariski open
	subset $V$ of $C$ such that $V\subseteq C\setminus S$ and $\varphi$
	extends to a morphism from $\bP^1_k\times_k V$ to itself. For every 
	$(\mu,\eta)\in \Z_{\geq 0}\times \N$, the equation
	$\varphi^{\mu+\eta}(z)=\varphi^{\mu}(z)$
	defines a Zariski closed subset $V_{\mu,\eta}$
	of $\bP^1_k\times_k V$ which is equidimensional of dimension 1. Define:
	$$U=V_{m,n}\setminus\left(\bigcup_{\mu=0}^{m-1} V_{\mu,n} \cup
	\bigcup_{p\mid n} V_{m,n/p}\right)$$
	where $p$ ranges over all prime factors of $n$. We have that
	$U$ is a Zariski open subset of $V_{m,n}$. Since $\varphi$ has
	a point of portrait $(m,n)$, the set $U$ is 
	non-empty.
	
	Let $\rho$ denote the projection from $\bP^1_k\times_k V$
	to $\bP^1_k$. We prove that the image $\rho(U)$
	cannot be a finite subset of $\bP^1_k$. Assume otherwise, say
	$\rho(U)=\{u_1,\ldots,u_r\}$. Then this implies that 
	all points of portrait $(m,n)$ under $\varphi$ are the
	constant points $u_1,\ldots,u_r$
	contradicting the assumption $(m,n)\notin W(\varphi)$. Hence $\rho(U)$
	is infinite. Since $\rho(U)$
	is constructible in $\bP^1_k$ by Chevalley's theorem, we must have that $\rho(U)$ 
	is
	co-finite in $\bP^1_k$.
	
	Now for every $\alpha\in \rho(U)$, pick any $P\in \rho^{-1}(\alpha)$,
	and let $\fp\in V$ be the image of $P$ under the projection
	from $\bP^1_k\times_k V$ to $V$. We have that $\alpha$ has portrait 
	$(m,n)$ under the action of $\varphi$ modulo $\fp$. This finishes the proof of Theorem~\ref{thm:projection}.
\end{proof}

We now prove of Theorem~\ref{thm:dual arbitrary portraits}, which in turn relies on Theorem~\ref{thm:projection}.

\begin{proof}[Proof of Theorem~\ref{thm:dual arbitrary portraits}.]
The proof uses an inductive process which is ``dual'' to the proof of Theorem \ref{arbitrary portraits}.
First, we let $k_0:=\overline{k(a_1,a_2,\dots,a_{d-2})}$, and we apply Theorem~\ref{thm:projection} to the polynomial 
$$f_0(z)=z^{d}+a_{d-2}z^{d-2}+\cdots + a_1z+a_0$$ 
defined over the function field $K_0:=k_0(a_0)$.  Then Lemma
\ref{lem:iso} shows that $f_0$ is non-isotrivial. By Lemma \ref{preperiodicity lemma}, there does not exist $c\in k$ which is preperiodic under $f_0$.
Therefore by Remark \ref{rem:dual main} and Kisaka's list \cite{Kisaka},
we have $W(f_0)=\emptyset$, hence $(m_0,n_0)\notin W(f_0)$.  
Take the exceptional set of of places $S_0$
to be the set containing only the place ``at infinity'' of $K_0=k_0(a_0)$
which is the only pole of $a_0$.
Hence we obtain the existence of the co-finite set $T^{(0)}\subset \bP^1(k)$ such that for all $c_0\in T^{(0)}$, there exists $a_{0,0}\in k_0$ such that $c_0$ has portrait $(m_0,n_0)$ under the action of $$f_1(z):=z^{d}+a_{d-2}z^{d-2}+\cdots +a_2z^2+a_{1}z+a_{0,0}.$$ 
This is the polynomial $f_0$ obtained after specializing $a_0$ to $a_{0,0}\in k_0$.  This specialization is equivalent with reducing $f_0$ modulo a place of $K_0$. Fix
$c_0\in T^{(0)}$ and a corresponding $a_{0,0}$.

The above process could be done inductively as follows. Let $i\in\{1,\dots, d-2\}$,  and assume we previously found co-finite sets $T^{(0)}$, $T^{(1)}$,\dots, $T^{(i-1)}\subset \bP^1(k)$ and 
fixed $c_j\in T^{(j)}$ for $0\leq j\leq i-1$ together with
$a_{0,i-1},a_{1,i-1},\ldots,a_{i-1,i-1}\in k_{i-1}$
where $k_{i-1}:=\overline{k(a_i,\ldots,a_{d-2})}$ such that the following hold.
Let $k_i:=\overline{k(a_{i+1},\ldots,a_{d-2})}$
with the understanding that $k_i=k$ when $i=d-2$, and let
$$f_i(z):=z^d+a_{d-2}z^{d-2}+\ldots+a_iz^i+a_{i-1,i-1}z^{i-1}+\ldots+a_{1,i-1}z+a_{0,i-1},$$ 
which is as a polynomial defined over the function field $K_i:=k_i(a_i,a_{i-1,i-1},\ldots,a_{0,i-1})$. Also, for each $j=0,\dots, i-1$, the point $c_j$ has portrait $(m_j,n_j)$ under the action of $f_i$. 

Lemma \ref{preperiodicity lemma} now 
asserts that there are at most $i$ constants $c\in k_i$
such that $\hhat_{f_i}(c)<\frac{1}{d}$. Since $c_0,\ldots,c_{i-1}$
are such constants and
since $(m_i,n_i)$
is distinct from $(m_0,n_0),\ldots,(m_{i-1},n_{i-1})$,
there exists no $c\in k$
such that $c$ has portrait $(m_i,n_i)$ under $f_i$. Therefore
we have $(m_i,n_i)\notin W(f_i)$.
Let $S_i$
be the set of places $\fp$ of $K_i$ such that 
$a_i$ has a pole, or for some $0\leq j\leq i-1$ the element $a_{j,i-1}$
has a pole, or for some $0\leq j\leq i-1$ the point $c_j$ does not have portrait $(m_j,n_j)$ under the action of $f_i$ modulo $\fp$. This set $S_i$
is finite by Lemma \ref{lem:generic portrait}. By Theorem \ref{MN}, 
there exists a co-finite set $T^{(i)}\subset \bP^1(k)$ such that for every $c_i\in T^{(i)}$
there exist $a_{0,i},\ldots,a_{i,i}\in k_i$ satisfying the following. For $0\leq j\leq i$, the point $c_j$ has portrait $(m_j,n_j)$
under the action of 
$$f_{i+1}(z):=z^{d}+a_{d-2}z^{d-2}+\ldots+a_{i+1}z^{i+1}+a_{i,i}z^i+\ldots+a_{1,i}z+a_{0,i}.$$
Continuing the above process until $i=d-2$,  we finish the proof
  of Theorem~\ref{thm:dual arbitrary portraits}.
\end{proof}

We conclude this section by proving Corollary~\ref{cubic}.

\begin{proof}[Proof of Corollary~\ref{cubic}.]
  As in the proof of Theorem~\ref{thm:dual arbitrary portraits}, let $k_1=\overline{k(a)}$, 
  $K_1=k_1(b)$. 
    We apply Theorem~\ref{thm:projection} to obtain 
  a co-finite subset $U^{(1)}$ of $k$ such that for every $c_1\in U^{(1)}$,
  the following holds.
  For every $(m_1,n_1)\in \Z_{\geq 0}\times \N$,
  there exists $\bar{b}\in \overline{k(a)}$
  such that $c_1$ has portrait $(m_1,n_1)$ under:
  $$\varphi_{c_1,m_1,n_1}(z):=z^3+az+\bar{b}$$
  regarded as a polynomial in $K_2[z]$. Here $K_2:=k(a,\bar{b})$
  is a function field over $k_2:=k$.
  
  We claim that $W(\varphi_{c_1,m_1,n_1})$ is empty. By Lemma~\ref{preperiodicity lemma}, $c_1$
  is the only constant preperiodic point of $\varphi_{c_1,m_1,n_1}$. Hence for
  every portrait $(m_2,n_2)\neq (m_1,n_1)$, we have $(m_2,n_2)\notin W(\varphi_{c_1,m_1,n_1})$.
  It suffices to show $(m_1,n_1)\notin W(\varphi_{c_1,m_1,n_1})$ by proving that
  $\varphi_{c_1,m_1,n_1}$ has a point $\gamma\neq c_1$ of portrait $(m_1,n_1)$. The case $n_1>1$ is easy: if 
  $m_1=0$ we pick $\gamma_1=\varphi_{c_1,m_1,n_1}(c_1)$, while if $m_1>0$ we pick
  $\gamma_1\neq c_1$ such that $\varphi_{c_1,m_1,n_1}(\gamma_1)=\varphi_{c_1,m_1,n_1}(c_1)$.
  Note that this is possible since $\varphi_{c_1,m_1,n_1}$ has no totally ramified point other than infinity.
  We now consider the case $n_1=1$. Since the polynomial:
  $$\varphi_{c_1,m_1,n_1}(z)-z=z^3+(a-1)z+\bar{b}$$
  is not the cube of a linear polynomial in $K_2[z]$, we have that $\varphi_{c_1,m_1,n_1}$
  has at least two distinct points $\alpha_1,\alpha_2$ having portrait $(0,1)$. Using the fact that $\varphi$
  has no totally ramified point (other than infinity) and looking at appropriate backward orbits
  of $\alpha_1$ and $\alpha_2$, we get at least 2 points having portrait $(m_1,1)$. Hence $W(\varphi_{c_1,m_1,n_1})=\emptyset$.

  
  
  Let $S(c_1, m_1, n_1)$ be the set of places $\fp$ of $K_2$ such that $\varphi_{c_1, m_1, n_1}$ has bad reduction at $\fp$
  or $c_1$ does not have portrait $(m_1, n_1)$ modulo $\fp$.  Then
  Applying Theorem~\ref{thm:projection}, we see then
  that for each $(m_2, n_2)$ there is a co-finite set $\cU(c_1, m_1, n_1, m_2,
  n_2)$ such that for all $c_2 \in \cU(c_1, m_1, n_1, m_2,
  n_2)$, there is a polynomial $f(z) = z^3 + \tilde{a}z + \tilde{b} \in k[z]$ such that,
  for $i = 1, 2$, $c_i$ has portrait $(m_i, n_i)$ under $f$. Define: 
\[ U^{(2)}(c_1):=\bigcap_{\left( (m_1, n_1), (m_2, n_2) \right)} \cU(c_1, m_1, n_1, m_2,
  n_2) \] 
 which is a co-countable subset of $k$. From our construction, the sets $U^{(1)}$ and
 $U^{(2)}(c_1)$ for every $c_1\in U^{(1)}$ satisfy the
 assertion in Corollary~\ref{cubic}.

  For the second assertion in the corollary, we simply pick the elements $(c_1,c_2)\in k^2$ satisfying
  $c_1\in U^{(1)}$ and $c_2\in U^{(2)}(c_1)$. 
\end{proof}
 


\section{Future directions}
\label{sect:future}

One might ask if something much stronger than Theorem~\ref{arbitrary
  portraits} and Corollary \ref{cubic} is true.  It is possible that
if $d \geq 3$, then for any distinct points $c_1, \dots, c_{d-1} \in
\bC$ and any (not necessarily distinct) $(m_1, n_1), \dots, (m_{d-1},
n_{d-1}) \in \Z_{\geq 0}\times \N$, there is a polynomial $f(z) = z^d +
a_{d-2} z^{d-2} + \dots + a_0 \in \bC[z]$ such that $c_i$ has portrait
$(m_i, n_i)$ under $f$.  We know of no counterexamples.  Note,
however, that in the case $d=2$, there is no polynomial $f(z) = z^2 + a_0$
such that 0 has portrait $(1,m)$; this follows immediately from the fact that
the general quadratic $z^2 + t$ ramifies completely at 0.  Since for
$d \geq 3$, the general degree $d$ polynomial in normal form has no
totally ramified points (other than the point at infinity), this particular
example has no analog in degree greater than 2.  On the other hand, it
is also true that here is no polynomial $f(z) = z^2 + a_0$ such that
$-1/2$ has portrait $(0,2)$.  We do not know if there are other
interesting cases of ``missing portraits'', either in degree 2 or in
higher degree.  In the case of degree 2, the techniques of this paper
reduce the problem to a finite computation, but we do not know if it
is feasible to carry out the computation in a reasonable amount of
time.  

One might also ask how much of Theorem~\ref{arbitrary portraits},
Theorem \ref{thm:dual arbitrary portraits}, and Corollary \ref{cubic}
holds in the more general context of nonconstant points in $\overline
  {k(a_0, \dots, a_{d-2})}$.  One issue that arises here is
  the possibility that some $c_i$ is an iterate of another $c_j$ under
  the general degree $d$ polynomial, something that cannot happen when
  all of the $c_i$ are in $k$.

The multi-portrait problem studied here was inspired by work of
Douady, Hubbard, and Thurston \cite{D-H}, who treated the problem of portraits
of critical points of rational functions.  Their work yielded not only
existence results, but also information about finiteness (up to change of
variables) and transversality (of intersections of hypersurfaces
corresponding to portraits of marked critical points).  We hope to
treat constant point analogs of these results in future work.




\end{document}